\documentclass[11pt, a4paper,leqno]{amsart}
\usepackage{amsmath,amsthm,amscd,amssymb,amsfonts,amsbsy}
\usepackage{latexsym}
\usepackage{exscale}

\usepackage[colorlinks=true, pdfstartview=FitV, linkcolor=blue, citecolor=red, urlcolor=blue, backref=page]{hyperref} 
\usepackage{latexsym}
\usepackage{tikz}
\usepackage{tkz-euclide}
\usetkzobj{all}
\usetikzlibrary{intersections}
\usetikzlibrary{shapes,arrows,calc}

\calclayout
\allowdisplaybreaks

\theoremstyle{plain}
\newtheorem{theorem}[equation]{Theorem}
\newtheorem{lemma}[equation]{Lemma}

\newtheorem{proposition}[equation]{Proposition}

\newtheorem{thmx}{Theorem}

\theoremstyle{definition}

\newtheorem{example}[equation]{Example}

\theoremstyle{remark}
\newtheorem{remark}[equation]{Remark}

\numberwithin{equation}{section}
\usepackage{chngcntr} 
\counterwithin{figure}{section}

\renewcommand{\emptyset}{\mbox{\textup{\O}}}

\newcommand\rh{\mathcal{H}}
\newcommand{\hd}{\mathcal{H}-\mbox{dim}\;}

\newcommand{\m}[1]{\mathcal{#1}}

\usepackage[margin=1.2in]{geometry}

\newcommand\rd{\mathrm{d}}

\newcommand{\dr}[1]{\mathrm{#1}}
\font\ursymbol=psyr at 10pt 
\def\urpartial{\mbox{\ursymbol\char"B6}}

\DeclareMathOperator*{\esssup}{ess\,sup}
\DeclareMathOperator*{\essinf}{ess\,inf}

\def\XXint#1#2#3{{\setbox0=\hbox{$#1{#2#3}{\int}$ }
\vcenter{\hbox{$#2#3$ }}\kern-.55\wd0}}

 \def\doublespaced{\baselineskip=\normalbaselineskip
    \multiply\baselineskip by 2}
\def\doublespace{\doublespaced}

\doublespace

\newcommand{\rn}[1]{\mathbb{R}^{#1}}

\newcommand{\ep}{\epsilon}
\newcommand{\es}{\emptyset}

\newcommand{\al}{\alpha}

\newcommand{\ga}{\gamma}

\newcommand{\de}{\delta}
\newcommand{\ds}{\displaystyle}

\newcommand{\rar}{\mbox{$\rightarrow$}}

\newcommand{\Ga}{\Gamma}
\newcommand{\la}{\lambda}
\newcommand{\La}{\Lambda}
\newcommand{\ar}{\partial}
\newcommand{\si}{\sigma}

\newcommand{\Om}{\Omega}
\newcommand{\be}{\beta}

\newcommand{\ph}{\phi}
\newcommand{\he}{\theta}

\newcommand{\Ph}{\Phi}

\newcommand{\sem}{\setminus}
\newcommand{\ze}{\zeta}

\newcommand{\ti}{\tilde}

\renewcommand*{\backref}[1]{}
\renewcommand*{\backrefalt}[4]{%
   \ifcase #1 (Not cited.)%
    \or        (Cited on page~#2.)%
    \else      (Cited on pages~#2.)%
    \fi}
\begin{document}

\allowdisplaybreaks

\title[$\sigma$-finiteness of elliptic measures]{$\sigma$-finiteness of elliptic measures for quasilinear elliptic PDE in space}


\address{Murat Akman\\
Instituto de Ciencias Matem\'aticas CSIC-UAM-UC3M-UCM\\
Consejo Superior de Investigaciones Cient{\'\i}ficas\\
C/ Nicol\'as Cabrera, 13-15\\
E-28049 Madrid, Spain} \email{murat.akman@icmat.es}


\address{John Lewis\\
Mathematics Department, University of Kentucky, Lexington, Kentucky, 40506
} \email{johnl@uky.edu}


\address{Andrew Vogel\\
Department of Mathematics, Syracuse University, Syracuse, New York 13244
} \email{alvogel@syracuse.edu}

\author{Murat Akman \and John Lewis \and Andrew Vogel}



\begin{abstract}
In this paper we study the Hausdorff dimension of a elliptic measure $\mu_{f}$ in space associated to a positive weak solution to a certain quasilinear elliptic PDE in an open subset and vanishing on a portion of the boundary of that open set. We show that this measure is concentrated on a set of $\sigma-$finite $n-1$ dimensional Hausdorff measure for $p>n$ and the same result holds for $p=n$ with an assumption on the boundary. 

We also construct an example of a domain in space for which the corresponding measure has Hausdorff dimension $\leq n-1-\delta$ for $p\geq n$ for some $\delta$ which depends on various constants including $p$. 

The first result generalizes the authors previous work in \cite{ALV13} when the PDE is the $p-$Laplacian and the second result generalizes the well known theorem of Wolff in \cite{W95} when $p=2$ and $n=2$.
\end{abstract}

\maketitle
\tableofcontents

\section{Introduction}
\label{intro}
In this paper we continue our study of the Hausdorff dimension of a measure associated with a certain positive weak solution, $u \geq 0,$ to a PDE of $ p $ Laplace type. To introduce the PDE and the measure, we fix $p$, $1<p<\infty,$ and let $f:\mathbb{R}^{n}\sem\{0\}\to (0,\infty)$ be a real valued function with the following properties,
\begin{align}
\label{prooff}
 \begin{array}{l}
a)\;f\; \mbox{is homogeneous of degree $p$}\;  \mbox{on}\; \mathbb{R}^{n}\setminus\{0\}.\\
\hspace{.8cm} \mbox{That is,}\; f(\eta)=|\eta|^{p}f\left(\frac{\eta}{|\eta|}\right)>0\; \mathrm{when}\; \eta\in\mathbb{R}^{n}\setminus\{0\}.\\ \\
b)\;   \m{D}f =(f_{\eta_{1}},\ldots, f_{\eta_{n}}) \; \mbox{has continuous partial derivatives when}\; \eta \neq 0. \\ \\
c)\;  f\; \mbox{is uniformly convex on}\; B(0,1)\setminus \bar{B}(0,1/2). \\
\hspace{.8cm}  \mbox{That is, there exists} \;  c_*\geq 1\; \mbox{such that for} \;  \eta\in\mathbb{R}^{n}, 1/2< |\eta|<1, \\
\hspace{.8cm}  \mbox{and all}\; \xi\in\mathbb{R}^{n}\; \mbox{we have}\; c_*^{-1} |\xi|^{2}\leq \sum\limits_{j,k=1}^{n} \frac{\partial^{2} f}{\partial \eta_{j}\eta_{k}}(\eta)\xi_{j}\xi_{k}\leq c_* |\xi|^{2}.\\
  \end{array}
\end{align}
Put $f(0)=0.$ We next give examples of such $f$.
\begin{example} 
From $a)$ in (\ref{prooff}) it follows that $ f(\eta)=\kappa (\eta)|\eta |^p$ when $\eta\in\rn{n}\sem\{0\}$, where $\kappa$ is homogeneous of degree 0. Using this fact one can show  that if $ \ep  $ is sufficiently small, then $f(\eta)=|\eta|^{p} (1+\epsilon\eta_{1}/|\eta|)$ satisfies (\ref{prooff}). Such an $f$ is not invariant under rotations.
\end{example}
From homogeneity of $f$ and Euler's formula we have for a.e $\eta\in\mathbb{R}^{n}$ that
\begin{align}
\label{fisdegreep}
\langle \m{D}f(\eta), \eta \rangle =p f(\eta) \; \mbox{and} \; \eta \, (\m{D}^{2} f(\eta)) =(p-1)\m{D}f(\eta)
\end{align}
where $\m{D}^{2} f(\eta)=(f_{\eta_{j}\eta_{k}})$ is an $n$ by $n$ matrix and $\eta$, $\m{D}f(\eta)$ are regarded as $1\times n$ row matrices.

Let $O$ be an open set in $\mathbb{R}^{n}$ and $\hat{z}\in\partial O$. Let $u$ be a positive weak solution in $ O\cap B(\hat{z}, \rho)$ to the Euler-Lagrange equation
\begin{align}
\label{flaplace}
\Delta_{f} u:=\nabla\cdot\m{D}f(\nabla u)=\sum\limits_{j,k=1}^{n}f_{\eta_{j}\eta_{k}}(\nabla u)u_{x_{k}x_{j}}=0
\end{align}
in $O\cap B(\hat{z},\rho)$. That is, $u\in W^{1,p}(O\cap B(\hat{z}, \rho))$ and
\[
\int\langle \m{D} f(\nabla u), \nabla\theta\rangle \rd x=0\, \, \, \mbox{whenever}\, \, \, \theta\in W^{1,p}_{0}(O\cap B(\hat{z},\rho))
\] 
where $ \nabla \theta (x)=(\frac{ \ar \theta }{\ar x_1}, \dots, \frac{ \ar \theta }{\ar x_n} ) (x) $ whenever these partials exist in the distributional sense. We assume also that $u$ has continuous zero boundary values on $\partial O\cap B(\hat{z},\rho)$. We continuously extend $u$ (denoted with $u$ also) to all $B(\hat{z},\rho)$ by setting $u\equiv 0$ in $B(\hat{z},\rho)\setminus O$. It is well known from \cite[Theorem 21.2]{HKM06} that there exists a positive locally finite Borel measure $\mu_{f}$ on $ \rn{n} $ associated with $u$. We call this measure as \textit{elliptic measure} associated with a positive weak solution of \eqref{flaplace}.  This measure has support contained in $\partial O\cap B(\hat{z},\rho)$  with the property that
\begin{align}
\label{ast}
\begin{split}
\int\langle\m{D} f(\nabla u), \nabla\phi\rangle \dr{d}x=-\int\phi\, \dr{d}\mu_{f}\; \mbox{whenever}\; \phi\in C_{0}^{\infty}(B(\hat{z},\rho)).
\end{split}
\end{align}
Existence of $\mu_{f}$ follows from the maximum principle, basic Caccioppoli inequalities for $u$ and the Riesz representation theorem for positive linear functional. Note that if $\partial O$ and $f$ are smooth enough then from an integration by parts in (\ref{ast}) and homogeneity in (\ref{fisdegreep}) we deduce that
\[
\dr{d}\mu_{f}=p\frac{f(\nabla u)}{|\nabla u|}\dr{d} \rh^{n-1}|_{\partial O\cap B(\hat{z},\rho)}.
\]
We next introduce the notion of the \textit{Hausdorff dimension of a measure}. To this end, let $\lambda$ be a real valued, positive, and increasing function on $(0,\infty)$ with $\lim\limits_{r\to 0}\lambda(r)=0$. For fixed $0<\delta$ and $E\subset \mathbb{R}^{n} $, we define $(\de, \la)-$\textit{Hausdorff content} of $E$ in the usual way;
\begin{align}
\label{hausdorffcontentdefn}
 \rh_{ \delta }^\lambda(E):= \inf\left\{\sum\limits_{i} \lambda(r_i)\; \mbox{where}\; E\subset \bigcup B(z_{i}, r_{i}),\; 0<r_{i}<\delta,\; x_{i}\in\mathbb{R}^{n}\right\}.
 \end{align}
Then the \textit{Hausdorff measure} of $E$ is defined by
\[
\rh^\lambda(E):= \lim\limits_{\delta \to 0 } \, \,  \rh_{ \delta }^\lambda(E).
 \]
In case $ \lambda(r) = r^\alpha $ we write $ \rh^\alpha $ for $ \rh^\lambda$. The Hausdorff dimension of $\mu_{f}$, denoted by $\hd{\mu_{f}}$, is defined by

\[
\hd{\mu_{f}}:=\inf\left\{\alpha:\; \exists\, \mbox{Borel set}\; E\subset\partial O\; \mbox{with}\; \rh^{\alpha}(E)=0\; \mbox{and}\; \mu_{f}(\mathbb{R}^{n}\setminus E)=0\right\}.
\] 
Recall that $\mu$ is said to be \textit{absolutely continuous} with respect to $\nu$ (if $\mu, \nu, $ are positive Borel measures) provided that $\mu(E)=0$ whenever $E$ is a Borel set with $\nu(E)=0$. Following standard notation, we write $\mu\ll\nu$. A set $E$ is said to have \textit{$\sigma-$finite} $\nu$ measure if
\[
E=\bigcup\limits_{i=1}^{\infty} E_{i}\; \mbox{with}\; \nu(E_{i})<\infty\; \mbox{for} \;i=1,\ldots,\infty.
\]
We note that if  $f(\eta)=|\eta|^{2}$, then the Euler-Lagrange equation in (\ref{flaplace}) is the usual \textit{Laplace equation}. In this case, if $u$ is the \textit{Green's function} for Laplace's equation with pole at some $z_{0}\in\Omega$, then the measure corresponding to this harmonic function $u$ as in (\ref{ast}) is \textit{harmonic measure} relative to $z_{0}$ and will be denoted by $\omega$.

The Hausdorff dimension of $\omega$ has been extensively studied in the last thirty five years in planar domains. In particular, in \cite{C85}, Carleson proved that
$\hd{\omega}=1$ when $\partial\Omega$ is a snowflake and $\hd{\omega}\leq 1$ for any self similar Cantor set. In \cite{M85}, Makarov proved that
\begin{thmx}[Makarov]
\label{makarov}
Let $\Omega$ be a simply connected domain in the plane and let $ \lambda(r):=r\, \mathrm{exp}\{A\sqrt{\log\frac{1}{r} \log\log\log\frac{1}{r}}\}$. Then
\begin{itemize}
\item[a)] $\omega$ is concentrated on a set of $\sigma-$finite $\rh^{1}$ measure,
\item[b)] $\omega\ll\rh^{\lambda}$ provided that $A$ is large enough.
\end{itemize}
\end{thmx}
We note  that Theorem \ref{makarov}  implies $\hd{\omega}=1$ when $\Omega$ is a simply connected domain. For arbitrary domains in the plane, in \cite{JW88}, Jones and Wolff proved that $\hd{\omega}\leq 1$ whenever $\Omega\subset\mathbb{R}^{2}$ and $\omega$ exists. In \cite{W93}, Wolff improved this result by showing that $\omega$ is concentrated on a set of $\sigma-$finite $\rh^{1}$ measure(see also \cite{B96,KW85,V93}).

The Hausdorff dimension of harmonic measure in higher dimensions is considerably less understood than in the plane. When $n\geq 3$, in \cite{B87}, Bourgain proved that $\hd{\omega}\leq n-\tau$, where $\tau>0$ depends only on the dimension $n$ and the exact value of $\tau$ remains unknown. On the other hand, in \cite{W95}, Wolff constructed examples in $\mathbb{R}^{3}$, we call {\it Wolff snowflakes}, for which the corresponding harmonic measures could have Hausdorff dimension either greater than $2$ or less than $2$. In \cite{LVV05}, the second author, Verchota, and the third author proved a conjecture of Wolff in the affirmative: it was shown that both sides of a Wolff snowflake in $\mathbb{R}^{n}$ could have harmonic measures, say $\omega_{1}, \omega_{2}$, with either $\min(\hd{\omega_{1}},\hd{\omega_{2}})>n-1\, \, \, \mbox{or}\, \, \, \max(\hd{\omega_{1}}, \hd{\omega_{2}})<n-1$.

If  $f(\eta)=|\eta|^{p}$ in (\ref{flaplace}), then the resulting PDE is called the \textit{$p-$Laplace equation}:
\begin{align}
\label{plaplace}
\nabla\cdot \left[ |\nabla u|^{p-2}\nabla u\right]=0.
\end{align}
In this case, a solution $u$ to (\ref{plaplace}) is called a \textit{$p-$harmonic function} and the corresponding measure in (\ref{ast}) associated with $u$ is called a \textit{$p-$harmonic measure} and will be denoted by $\mu_{p}$.

The nonlinearity and degeneracy of the $p-$Laplace equation makes it difficult to study the Hausdorff dimension of $p-$harmonic measure.  The first result was obtained in \cite{BL05}, when Bennewitz and the second author studied the Hausdorff dimension of a $p-$harmonic measure, associated with a positive $p-$harmonic function $u$ in $ N \cap \Om \subset \rn{2} $ with continuous boundary value 0 on $\ar\Om$. In that result $\ar\Om$ is a quasicircle and $N$ is an open neighborhood of $\ar\Om$. It was shown that all such measures, $\mu_p$, corresponding to $u$, $\Om$, $p$ as above, have the same Hausdorff dimension. Moreover,
\begin{align*} 
\hd{\mu_p}\geq 1\; \mbox{when}\; 1<p<2\; \mbox{ while } \hd{\mu_p}\leq 1\; \mbox{when}\; p>2. 
\end{align*} 
After earlier studies in \cite{BL05,L06,LNP11}, the second author proved the following analogue of Theorem \ref{makarov} in the $p-$harmonic setting (see \cite{L12});
\begin{thmx}[Lewis]
\label{lewis}
Assume that $\Omega\subset\mathbb{R}^{2}$ is a bounded simply connected domain and $N$ is a neighborhood of $\partial\Omega$. Let $u$ be a positive p-harmonic in $\Omega\cap N$ with zero continuous boundary values on $\partial\Omega$. Let $\mu_{p}$ be the $p-$harmonic measure associated with $u$ as described above. Let $ \lambda(r)$ be as in Theorem \ref{makarov}. Then
\begin{itemize}
\item[a)] If $1<p<2$, there exists $A=A(p)\geq 1$, such that $\mu_{p}\ll \rh^{\lambda}$.
\item[b)] If $2<p<\infty$, then $\mu_{p}$ is concentrated on a set of $\sigma-$finite $\rh^{1}$ measure.
\end{itemize}
\end{thmx}
A key fact used in \cite{BL05,L06,L12,LNP11} is that if $\zeta=u$ or $\zeta=u_{x_{i}}$, $i=1, 2$, then $\zeta$ is a weak solution to
\begin{align}
\label{Lzeta}
L\zeta=\sum\limits_{j,k=1}^{2}\frac{\urpartial }{\urpartial x_{k}}\left(b_{jk} \frac{\urpartial \zeta}{\urpartial x_{j}}\right)=0\; 
\end{align}
where
\[
b_{jk}=|\nabla u|^{p-4}\left[ (p-2)u_{x_{j}}u_{x_{k}}+\delta_{jk}|\nabla u|^{2}\right].
\]
Furthermore, if $v=\log|\nabla u|$ then $Lv\leq 0 (Lv\geq 0)$ when $1<p\leq 2 (2\leq p<\infty)$. Moreover, arguments in these papers also make heavy use of the fundamental inequality;
\begin{align} 
\label{funeq}
\frac{\hat{u}(z)}{d(z,\partial\Omega)}\approx |\nabla \hat{u}(z)|\; \mbox{whenever}\; z\in \Omega\setminus \bar{B}(z_{0}, r_{0}).
\end{align}
where $\hat{u}$ is a certain ``$p-$capacitary function'' in $\Omega\setminus \bar{B}(z_{0}, r_{0})$ for some fixed $z_{0}\in\Omega$ and $r_{0}=d(z_{0}.\partial\Omega)/2$. The proof of (\ref{funeq}) is highly nontrivial in a simply connected domain when $1<p\neq 2<\infty$, and in fact is the main result proved by the second author, Nystr\"{o}m, and Poggi-Corradini in \cite[Theorem 1.5]{LNP11}. However if $p = 2$, (\ref{funeq}) is an easy consequence of the Koebe distortion estimates for a univalent function (use $\hat u$ = a Green's function for $\Omega$). We also note that (\ref{funeq}) can easily fail in arbitrary domains of $ \rn{n}$ for $n\geq 2$.

Tools developed for p-harmonic functions in a series of papers by the second author and Nystr\"{o}m were used in \cite{LNV11} to obtain that $\mu_{p}$ is concentrated on a set of $\sigma-$finite $\rh^{n-1}$ measure when $\partial\Omega \subset \rn{n} $ is sufficiently flat in the sense of Reifenberg, $u>0$ is $p$ harmonic near $\ar\Om$ and $p\geq n$.  It was also shown in the same paper that if $p\geq n$ then all examples produced by Wolff's method had $\hd{\mu_{p}}<n-1$, while if $p>2$, was near enough $2$, then there existed a Wolff snowflake for which $\hd{\mu_{p}}>n-1$. These examples provided the current authors with the necessary intuition to state and prove the following theorem in \cite{ALV13}.
\begin{thmx}[Akman, Lewis, Vogel]
\label{alv13}
Let $O \subset\mathbb{R}^{n}$ be an open set and $\hat{z}\in \partial O$, $\rho>0$. Let $u>0$ be p-harmonic in $O\cap B(\hat{z},\rho)$ with continuous zero boundary values on $\partial O\cap B(\hat{z},\rho)$, and let $\mu_{p}$ be the p-harmonic measure associated with $u$. If $p>n$ then $\mu_{p}$ is concentrated on a set of $\sigma-$finite $\rh^{n-1}$ measure. If $p=n$ the same conclusion is valid provided $\partial O\cap B(\hat{z}, \rho)$ is locally uniformly fat in the sense of $n-$capacity.
\end{thmx}

The definition of a \textit{locally uniformly fat set} will be given in section \ref{preplemmas}. We remark that Theorem \ref{alv13} and the definition of $\hd{\mu_{p}}$ imply that $\hd{\mu_{p}}\leq n-1$ for $p\geq n$. A key lemma proved in this paper states  that if $v=\log |\nabla u|$, then $Lv\geq 0, $ weakly on  $ \{ x : \nabla u ( x ) \not = 0 \}, \, $  when $p\geq n.$  Here $L$ is defined as in (\ref{Lzeta}) with 2 replaced by $n$ in the summation. Using this fact, some basic estimates for $p$ harmonic functions, and  a stopping stopping time argument as in \cite{JW88,W93}, we eventually arrived at Theorem \ref{alv13}.

In \cite{ALV12}, the authors studied the PDE (\ref{flaplace}), $\Delta_{f} u=0$, and showed in $\mathbb{R}^{2}$ that if $ u, f $ are sufficiently smooth and $ \nabla u (x )  \not = 0, $ then both $u$, $u_{x_{i}}$, $i=1,2$, satisfy
\begin{align}
\label{Ltilde}
\tilde{L}\zeta\, := \, \sum\limits_{k,j=1}^{2}\frac{\urpartial }{\urpartial x_{k}}\left(f_{\eta_{j}\eta_{k}}(\nabla u) \frac{\urpartial \zeta}{\urpartial x_{j}}\right)=0.
\end{align}
in an open neighborhood of $x$. Furthermore, if $\tilde{v}=\log f(\nabla u)$  then pointwise in this neighborhood $\tilde{L}\tilde{v}\leq 0 (\tilde{L}\tilde{v}\geq 0)$ when $1<p\leq 2 (2\leq p<\infty)$. In \cite{A13} it was shown by the first author for general $f$ as in \ref{prooff} that $\tilde{L}\tilde{v}\leq 0 (\tilde{L}\tilde{v}\geq 0)$ weakly when $1<p\leq 2 (2\leq p<\infty)$. Using this fact and following the game plan of \cite{BL05,LNP11}, the first author proved in the same paper that
\begin{thmx}[Akman]
\label{akman}
Let $\Omega\subset\mathbb{R}^{2}$ be any bounded simply connected domain and let $N$ be a neighborhood of $\partial\Omega$. Let $u$ be a positive weak solution to (\ref{flaplace}) in $\Omega\cap N$ with zero continuous boundary values on $\partial\Omega$. Let $\mu_{f}$ be the measure associated with $u$ as described above.
Let $\tilde{\lambda}(r):=r\,\exp\{A\sqrt{\log\frac{1}{r}\,\log\log\frac{1}{r}}\}\; \mathrm{for}\; 0<r<10^{-6}$. Then
\begin{itemize}
\item[a)] If $1<p\leq 2$, there exists $A=A(p, f)\geq 1$, such that $\mu_{f}\ll \rh^{\tilde{\lambda}}$.
\item[b)] If $2\leq p<\infty$, there exists $A=A(p, f)\leq -1$ such that $\mu_{f}$ is concentrated on a set of $\sigma-$finite $\rh^{\tilde{\lambda}}$ measure.
\end{itemize}
\end{thmx}
Note that  Theorem \ref{akman} implies
\[
\begin{aligned}
\hd{\mu_{f}} \left\{
\begin{array}{ll}
\geq 1 & \, \mbox{when}\, \, 1<p< 2,\\
=1 & \, \mbox{when}\, \, p=2,\\
\leq 1 & \, \mbox{when}\, \, 2< p<\infty.
\end{array}
\right.
\end{aligned}
\] We also note that
Theorem \ref{akman} is slightly weaker than Theorem \ref{makarov} when $f(\eta)=|\eta |^{2}$, $\mu_{f}=\omega$, and Theorem \ref{lewis} when $f(\eta)=|\eta|^{p}$, $1<p\neq 2<\infty$, $\mu_{f}=\mu_{p}$.

In this paper, we focus on the Hausdorff dimension of $\mu_{f}$, in  the same setting as in Theorem \ref{alv13}.  More specifically we prove
\begin{theorem}
\label{alv13.1}
Let $O \subset\mathbb{R}^{n}$ be an open set and $\hat{z}\in \partial O$, $\rho>0$. Let $f$ be as in (\ref{prooff}). Let $u>0$ be a weak solution to $\Delta_{f} u=0$ (see \ref{flaplace}) in $O\cap B(\hat{z},\rho)$ with continuous zero boundary values on $\partial O\cap B(\hat{z},\rho)$, and  let $\mu_{f}$ be the measure associated with $u$ as in (\ref{ast}). \\
If $p>n$ then $\mu_{f}$ is concentrated on a set of $\sigma-$finite $\rh^{n-1}$ measure. The same result holds when $p=n$ provided that $\partial O\cap B(\hat{z}, \rho)$ is locally uniformly fat in the sense of $n-$capacity.
\end{theorem}
\begin{remark}
Theorem \ref{alv13.1} and the definition of the Hausdorff dimension of a measure imply once again that $\hd{\mu_{f}}\leq n-1$ when $p\geq n$.
\end{remark}
We also construct for a given $ f $ some domains in $\mathbb{R}^{n}$ for which $\hd{\mu_{f}} < n-1$ when
$p\geq n$. To give the construction, let $0<\alpha<\beta<1/2$ be fixed numbers and let $S$ be the cube in $\mathbb{R}^{n}$ with side length $1$ and centered at $0$. Let $S'$ be the cube with side length $a_0=1/2$ and centered at $0$ and set $\mathcal{C}_{0}=S'$. Let $\tilde{Q}_{1,1},\ldots, \tilde{Q}_{1,2^{n}}$ be the closed corner cubes of $\mathcal{C}_{0}$ of side length $a_{0}a_{1}$, $\alpha\leq a_{1}\leq\beta$. Let $\mathcal{C}_{1}=\bigcup\limits_{i=1}^{2^{n}}\tilde{Q}_{1,i}$. Let $\{\tilde{Q}_{2,j}\}$, $j=1,\ldots, 2^{2n},$ be the closed corner cubes of each $\tilde{Q}_{1,i}$, $i=1,\ldots, 2^{n}$ of side length $a_{0}a_{1}a_{2}$, $\alpha\leq a_{2}\leq \beta$. Let $\mathcal{C}_{2}=\bigcup\limits_{j=1}^{2^{2n}}\tilde{Q}_{2,j}$ (see figure \ref{forucornercantorset}).

\begin{figure}[!ht]
\centering
\begin{tikzpicture}[scale=.7]
\begin{scope}[shift={(-10,0)}]
\draw[fill=gray, opacity=.5] (0,0) rectangle (4,4);
\node[below, label={[xshift=1.5cm, yshift=-1cm]$\mathcal{C}_{0}$}] (4,-2) {};
\end{scope}
\begin{scope}[shift={(-5,0)}]
\foreach \i in {0,3}
\foreach \j in {0,3}
{
\draw[fill=gray, opacity=.5] (\i, \j) rectangle (1+\i, 1+\j);
}
\node[below, label={[xshift=1.5cm, yshift=-1cm]$\mathcal{C}_{1}$}] (4,-2) {};
\end{scope}
\begin{scope}[shift={(0,0)}]
\foreach \i in {0,.66,3,3.66}
\foreach \j in {0,.66,3,3.66}
{
\draw[fill=gray, opacity=.5] (\i, \j) rectangle (.3+\i, .3+\j);
}
\node[below, label={[xshift=1.5cm, yshift=-1cm]$\mathcal{C}_{2}$}] (4,-2) {};
\end{scope}
\begin{scope}[shift={(5,0)}]
\foreach \j in {0,.66,3,3.66}
\foreach \k in {0,.66,3,3.66}
{
\begin{scope}[shift={(\j,\k)}]
\foreach \i in {1,...,10}
  \fill[red] (rnd*.33, rnd*.33) circle (.75pt);
  \end{scope}
 } 
 \node[below, label={[xshift=1.5cm, yshift=-1cm]$\mathcal{C}$}] (4,-1) {};
\end{scope}
\end{tikzpicture} 
\caption{The sets $\m{C}_{0}, \m{C}_{1}, \m{C}_{2}, \m{C}$ when $n=2$.}
\label{forucornercantorset}
\end{figure}

Continuing recursively, at the $m$ th step we get $2^{nm}$ closed cubes, $\tilde{Q}_{m, j}$, $j=1, \ldots, 2^{nm}$, of side length $a_{0}a_{1}a_{2}\ldots, a_{m}$, $\alpha\leq a_{m} \leq \beta$. Let $\mathcal{C}_{m}=\bigcup\limits_{j=1}^{2^{nm}}\tilde{Q}_{m,j}$. Then $\mathcal{C}$ is obtained as the limit in the Hausdorff metric of $\mathcal{C}_{m}$ as $m\to\infty$.

Following an unpublished result of Jones and Wolff (see \cite[Chapter IX, Theorem 2.1]{GM05}), we prove
\begin{theorem}
\label{alv13.2}
Let $S$ be the unit cube and $\m{C}$ be the set constructed above. Let $u^{\infty}$ be a positive weak solution to (\ref{flaplace}) for fixed $p\geq n$ in $S\setminus \m{C}$ with boundary values $1$ on $\partial S$ and $0$ on $\m{C}$. Let $\mu^{\infty}_{f}$ be the measure associated with $u^{\infty}$ as in (\ref{ast}).\\ 
Then $ \hd{\mu^\infty_f}\leq n -1-\de$ for some $ \de=\de (p,n, c_*,\al,\be,f)>0$. 

Moreover, $ \de \geq c^{-1} (p- n) $ where $ c \geq 1 $ can be chosen to depend only on $n, \al, \be$, and $c_*$ in (\ref{prooff}) when $p\in [n, n+1]$.  

If $f=g^p$ where $ g $ is homogeneous of degree $1$, uniformly convex, and has continuous second partials, then $\de$ can be chosen independent of $p\in [n, n+1]$, so depends only on $n, \al, \be, g$.   
\end{theorem}

In what follows, we state some regularity results for $u$ in section \ref{preplemmas}. In section \ref{subsol}, we show that $\log f(\nabla u)$ is a weak sub solution to $\tilde{L}$ when $p\geq n$ where $\tilde{L}$ is as in (\ref{Ltilde}) with $2$ replaced by $n$ in the summation. In section \ref{advancedregres} we prove more advanced regularity results and essentially begin the proof of Theorem \ref{alv13.1}. In section \ref{mainproof}, we prove a proposition and finish the proof of Theorem \ref{alv13.1}. In section \ref{section5}, we prove Theorem \ref{alv13.2}.

In general to prove Theorem \ref{alv13.1} we follow the proof of Theorem \ref{alv13} which in turn made effective use of the proof scheme in
\cite{JW88,W93}. However the proof that $ \log f(\nabla u)$ is a weak sub solution to $\tilde{L}$
is more involved, and in fact somewhat surprising to us, than the corresponding proof for $ f ( \nabla u ) = |\nabla u |^p, $
since in this case we could use rotational invariance of the $p$ Laplace equation to considerably simplify the calculations. Also regularity results for $ u, \nabla u, \log f ( \nabla u ), $ require more care than in \cite{ALV13} due to the nearly endpoint structural assumptions on $ f $ in  \eqref{prooff}.

Likewise to prove Theorem \ref{alv13.2}, we use the proof scheme in \cite[chapter IX]{GM05} only now we have little control over the zeros of $ \nabla u$. This lack of control forces us into an alternative finess type argument which produces the `hodge podge' of results on $ \de $ in Theorem \ref{alv13.2}, rather than what we hoped to prove, namely  $ \de > a > 0 $ on $ [n, n + 1]$ (provided $ c_* $ in \eqref{prooff} is constant for $ p \in [n, n+1]$).
\section{Notation and Preparatory Lemmas}
\label{preplemmas}
Let $x=(x_{1}, \ldots, x_{n})$ denote points in $\mathbb{R}^{n}$ and let $\overline{E},$ $\partial E,$ be the closure and boundary of the set $E\subset\mathbb{R}^{n}.$ Let $\langle \cdot, \cdot\rangle$ be the usual inner product in $\mathbb{R}^{n}$ and $|x|^{2}=\langle x, x\rangle.$  Let $d(E,F)$ denote the distance between the sets $E$ and $F$. Let $B(x, r)$ be the open ball centered at $x$ with radius $r > 0$ in $\mathbb{R}^{n}$ and let $dx$ denote Lebesque $n-$measure in $\mathbb{R}^{n}$. Given $O'$ an open set $ \subset \mathbb{R}^{n}$ and $q, 1 \leq q \leq \infty$, let $W^{1,q}(O')$
denote equivalence classes of functions $h:\mathbb{R}^{n}\to\mathbb{R}$ with distributional gradient $\nabla h =\langle h_{x_1}, \ldots, h_{x_n}\rangle$, both of which are $q$ th power integrable on $O'$ with Sobolev norm

\begin{align*}
\|h\|^{q}_{W^{1,q}(O')}=\int\limits_{O'}(|h|^{q}+|\nabla h|^{q})\rd x.
\end{align*}
Let $C^{\infty}_{0}(O')$ be the set of infinitely differentiable functions with compact support in $O'$ and let $W^{1,q}_{0}(O')$ be the closure of $C^{\infty}_{0}(O')$ in the norm of $W^{1,q}(O')$.

Let $K\subset \overline{B}(x,r)$ be a compact set and let $\mathfrak{A}:=\{\phi\in W^{1,n}_{0}(B(x,2r)):\; \phi\equiv 1\; \mbox{on}\; K\}$. We let
\begin{align}
\label{unfat}
\mbox{Cap}(K, B(x,2r)):=\inf\limits_{\phi\in \mathfrak{A}}\; \int\limits_{\mathbb{R}^{n}}|\nabla \phi|^{n}\rd x.
\end{align}
We say that a compact set $K\subset\mathbb{R}^{n}$ is \textit{locally $(n, r_{0})$ uniformly fat} or \textit{locally uniformly $(n,r_{0})$ thick} provided there exists $r_{0}$ and $c$ such that whenever $x\in K$ and $0\leq r \leq r_{0}$,
\[
\mbox{Cap}(K \cap \overline{B}(x,r), B(x,2r))\geq c>0.
\]

In the sequel, $c$ will denote a positive constant $\geq 1$ (not necessarily the same at each occurrence), which may depend only on $ p, n, c_*$ unless otherwise stated. In general, $c(a_1, \dots, a_n)$ denotes a positive constant $\geq 1$ which may depend only on $p, n, c_*, a_1, \dots, a_n$ not necessarily the same at each occurrence.  $A\approx B$ means
that $ A/B $ is bounded above and below by positive constants depending only on $p, n, c_*$. 

In this section, we will  always assume that $ 2 \leq n \leq p < \infty$, and $r>0$. We also assume that $\tilde{O}$ is an open set in $\mathbb{R}^{n}$ and $w\in\partial\tilde{O}$.

We begin by stating some interior and boundary estimates for a positive weak solution $\tilde{u}$ to (\ref{flaplace}) in $\tilde{O}\cap B(w, 4r)$. If $p=n$, we assume $\partial\tilde{O}\cap \overline{B}(w,4r)$ is $(n, r_{0})$ uniformly fat as defined above using the capacity in (\ref{unfat}). We assume that $\tilde{u}$ has zero boundary value on $\partial\tilde{O}\cap B(w,4r)$ in the Sobolev sense and we extend $\tilde{u}$ as above by putting $\tilde{u}\equiv 0$ on $B(w,4r)\setminus \tilde{O}$. Then as in (\ref{ast}) let $\tilde{\mu}_{f}$ be the positive Borel measure corresponding to $\ti u$.

References for the proofs of Lemmas \ref{fislegthanu}-\ref{localholderfornablau} can be found in \cite{ALV13} where these lemmas are stated for $f(\eta)=|\eta|^p$, however they also hold for $f$ as in \eqref{prooff}. Let $c_*$ be as in \eqref{prooff}.
\begin{lemma} 
\label{fislegthanu}
Let $\tilde{O}, w,r,\tilde{u}, f, \tilde{\mu}_{f}$ be as above. Then there exists constant $c=c(p, n, c_*) $ such that
\begin{align*}
&\frac{1}{c}\, r^{p-n}\int\limits_{B(w,\frac{r}{2})} f(\nabla \tilde{u})\rd x \leq \esssup\limits_{B(w,r)}\tilde{u}^{p}\leq c\, \frac{1}{r^{n}}\int\limits_{B(w,2r)}\tilde{u}^{p}\rd x.
\end{align*}
If $B(z,2r')\subset\tilde{O}\cap B(w,4r)$ for some $r'>0$ then there is a constant $c=c(p,n,c_*)$ such that
\begin{align*}
\esssup\limits_{B(z,r')} \tilde{u} \leq c\, \essinf\limits_{B(z,r')} \tilde{u}.
\end{align*}
\end{lemma}
\begin{lemma}
\label{uisholder}
Let $\tilde{O}, w, r, \tilde{u}, f$ be as in Lemma \ref{fislegthanu}. Then there is $\alpha'=\alpha'(p, n, c_*)\in (0,1)$ and $ c = c (p, n, c_*),$ such that $\tilde{u}$ has a H\"older continuous representative in $B(w,4r)$ (also denoted $ \ti u$). If $\tilde{w}, \hat{w}\in B(w,r)$ then
\begin{align*}
 |\tilde{u}(\tilde{w})-\tilde{u}(\hat{w})|\leq c\, \left(\frac{|\tilde{w}-\hat{w}|}{r}\right)^{\alpha'} \esssup\limits_{B(w, 2r)} \tilde{u}\, .
\end{align*}
\end{lemma}
\begin{lemma}
\label{uismu}
Let $\tilde{O}, w, r, \tilde{u}, f, \tilde{\mu}_{f}$ be as in Lemma \ref{fislegthanu}. Then there exists $c=c(p, n, c_*)\geq 1$ such that
\begin{align*}
\tfrac{1}{c}\, r^{p-n} \tilde{\mu}_{f}(B(w, \tfrac{r}{2})) \leq (\esssup\limits_{B(w,r)} \tilde{u})^{p-1} \leq c\, r^{p-n}\tilde{\mu}_{f}(B(w, 2r)).
\end{align*}
\end{lemma}
\begin{remark}
The left-hand side of the inequality in Lemma \ref{uismu} is true for any open $\tilde{O}$ and $p \geq n$. However, the right-hand side of this inequality requires uniform fatness when $p=n$ and that is the main reason why the uniform fatness assumption appears in Theorem \ref{alv13.1}.
\end{remark}
\begin{lemma}
\label{localholderfornablau}
Let $\tilde{O}, w, r, \tilde{u}, f$ be as in Lemma \ref{fislegthanu}. Then $ \ti u$ has a representative in $W^{1,p}(B(w,4r))$ with H\"older continuous
derivatives in $\tilde{O}\cap B(w,4r)$. In particular, there exists $\alpha''$, $0<\alpha''<1$, and $c\geq 1$, depending only on $p, n, c_*,$ with
\begin{align*}
\begin{split}
|\nabla \tilde{u}(x)-\nabla \tilde{u}(y)| \leq c \left(\frac{|x-y|}{\hat{r}}\right)^{\alpha''}\, \esssup\limits_{B(\tilde{w},\hat{r})}|\nabla \tilde{u}| \leq
\frac{c}{\hat{r}}\left(\frac{|x-y|}{\hat{r}}\right)^{\alpha''} \esssup\limits_{B(\tilde{w},\hat{r})}\tilde{u}.
\end{split}
\end{align*}
whenever $x,y\in B(\tilde{w}, \hat{r}/2)$, and $B(\tilde{w}, 4\hat{r})\subset \tilde{O}\cap B(\tilde{w}, 4r)$.

Moreover,
\begin{align*}
\int\limits_{B(\tilde{w},\hat{r})}|\nabla \tilde{u}|^{p-2}\sum\limits_{k,j=1}^{n}(\tilde{u}_{x_{k}x_{j}})^{2}\rd x \leq \frac{c}{ \hat{r}^{2}}\int\limits_{B(\tilde{w},2\hat{r})}|\nabla \tilde{u}|^{p}\rd x.
\end{align*}
\end{lemma}

\begin{lemma}
\label{ballminusball}

Let $\tilde{O}, w, r, \tilde{u}$ be as in Lemma 2.4. Suppose   for some  $ z \in \rn{n},  t \geq  100 r,    $   that  $ w \in
\partial B ( z, t ) $   and
\[
B ( w, 4r ) \sem \bar B ( z, t ) =  B ( w, 4 r ) \cap \tilde{O}.
\]
Then there exists $\alpha'''= \alpha''' (p, n, f) \in (0,1) $ for which  $\tilde{u} |_{ \tilde{O} \cap B(w,3r)}    $ has  a  $ C^{ 1, \alpha'''} \cap W^{1,p} $  extension to  the closure of    $    B ( w, 3r)  \sem  \bar B( z, t )  $ (denoted $\bar{u}$).  
  Moreover, 
  \[
  \int\limits_{ \tilde{O} \cap B (  w ,  r/2  ) \cap \{ | \nabla \bar{u}  | > 0 \}  } \, | \nabla  \bar{u}  |^{p-2}   \, \sum_{j,k = 1}^n \bar{u}_{x_j x_k}^2  d x   \leq \, \frac{c}{r^{2}} \int\limits_{ \tilde{O} \cap B (w ,  2r) } \, | \nabla \bar{u} |^{p} \, d x
       \]
 and if $ y, \tilde{y} \in  \tilde{O}  \cap B (  w,   r/2 ),  $  then
    \[  \begin{array}{l} 
    \frac{1}{c}\, |\nabla\bar{u}(y)-\nabla\bar{u}(\tilde{y} )|\leq \left(\frac{|y - \tilde{y} |}{ r}\right)^{\alpha'''} \, 
\max_{\tilde{O} \cap \bar B (w, r)} \, |\nabla \bar{u}|        \leq \, \frac{c}{r}\,   \left(\frac{|y  - \tilde{y} |}{r}\right)^{\alpha'''} \, \max_{\tilde{O} \cap B (w, 2r )} \bar{u}.
    \end{array} \] 
\end{lemma}
  \begin{proof}   
Lieberman  in \cite{lieberman} essentially   proves   the   above  lemma.   A  careful  reading of  his paper  gives the second  estimate  in this  lemma as well as the fact that   $  | \nabla \bar u |   \geq  c^{-1}  $   in $  B ( \ze,  r/c )  $  whenever  $  \zeta \in  \ar B ( z, t )  \cap   B ( w, 7r/2) $   where  $ c \geq 1 $  depends  only on $ p, n, $ and the structure constants for  $ f.  $   The first estimate   then  follows  from H\"{o}lder continuity of  derivatives,    the fact that  derivatives of  $  \bar u  $     satisfy  a  uniformly elliptic  PDE  in  divergence form   near     $  \ar  B (  z, t ) \cap    B ( w, 3r )     $  (see (\ref{intbypartszlsoln})), and  a  Caccioppoli  inequality.  
\end{proof}

\section{Sub solution estimate} 
\label{subsol}
Let $\tilde{L}$ be defined as in (\ref{Ltilde}) with $2$ replaced by $n$ in the summation. That is,
\begin{align}
\label{Lzeta0}
\tilde{L}\zeta\, = \, \sum\limits_{k,j=1}^{n}\frac{\urpartial }{\urpartial x_{k}}\left(f_{\eta_{j}\eta_{k}}(\nabla \tilde{u}) \frac{\urpartial \zeta}{\urpartial x_{j}}\right).
\end{align}
Let  $ \ti{v}(x)=\log f(\nabla \ti{u}(x))$ for $x\in\tilde{O}\cap B(w, 4r)$. In this section we first show that $\tilde{L}\ti{v}\geq 0$ weakly  in a domain $ \Om\subset \tilde{O} \cap B (w, 4r)$ when $p\geq n$ and  $\nabla\ti{u}\neq 0$ in $\Om$.  To do so we note that  Lemma \ref{localholderfornablau} implies $ \ti{u}$ is locally in $W^{2,2}(\Om)$ so (\ref{flaplace}) holds almost everywhere in $\Om$. It follows that for $l=1,2,\ldots, n$, 
\begin{align}
 \label{intbypartszlsoln}
\begin{split}
0&=\int\limits_{\Omega}\langle \m{D}f( \,\nabla \tilde{u}), \nabla \phi_{x_{l}}\rangle \rd x
=-\int\limits_{\Omega}\sum\limits_{k=1}^{n} \frac{\urpartial (f_{\eta_{k}}(\nabla \tilde{u}))}{\urpartial x_{l}}\, \phi_{x_{k}} \rd x \\
&=-\int\limits_{\Omega}\sum\limits_{k,j=1}^{n} f_{\eta_{k}\eta_{j}}(\nabla \tilde{u})(\ti u_{x_{l}})_{x_{j}} \phi_{x_{k}} \rd x.
\end{split}
\end{align}
whenever $\phi\in C^{\infty}_{0}(\Om)$ and non-negative. Therefore, $\zeta=\tilde{u}_{x_{l}}$, $l=1,\dots, n$, is a weak solution to \eqref{Lzeta0}. From \eqref{fisdegreep} we also have
\begin{align}
\label{Luiszero}
\begin{split}
\int\limits_{\Omega}\sum\limits_{k,j=1}^{n}f_{\eta_{j}\eta_{k}}(\nabla \tilde{u})\tilde{u}_{x_{j}} \phi_{x_{k}}\, \rd x=(p-1)\int\limits_{\Omega}\sum\limits_{k=1}^{n} f_{\eta_{k}}(\nabla \tilde{u}) \phi_{x_{k}}\rd x=0.
\end{split}
\end{align}
From (\ref{Luiszero}) we deduce that $\zeta=\tilde{u}$ is also a weak solution to (\ref{Lzeta0}).
Let $\mathfrak{b}_{kj}=f_{\eta_{k}\eta_{j}}(\nabla \tilde{u})$ and observe that
 for almost every $ x \in \Om, $ where $ \nabla \ti u ( x ) \neq 0, $ 
 \begin{align}
 \label{bijvij}
\mathfrak{b}_{kj}\tilde{v}_{x_{j}}=\frac{\mathfrak{b}_{kj}}{f(\nabla \tilde{u})}\sum\limits_{m=1}^{n}f_{\eta_{m}}(\nabla \tilde{u})\tilde{u}_{x_{m}x_{j}}.
\end{align}
Using (\ref{bijvij}) we find that

\begin{align}
\label{IandII}
 \begin{split}
\int\limits_{\Omega}\sum\limits_{k,j=1}^{n}\mathfrak{b}_{kj}\tilde{v}_{x_{j}}\phi_{x_{k}}\, \rd x
&=\int\limits_{\Omega}\sum\limits_{k,j=1}^{n}\frac{\mathfrak{b}_{kj}}{f(\nabla \ti u)}\sum\limits_{m=1}^{n}f_{\eta_{m}}(\nabla \tilde{u})\tilde{u}_{x_{m}x_{j}}\phi_{x_{k}}\, \rd x\\
&=-\int\limits_{\Omega}\sum\limits_{m,k,j=1}^{n}\frac{\urpartial }{\urpartial x_{k}}\left(\frac{f_{\eta_{m}}(\nabla \tilde{u})}{f(\nabla \tilde{u})}\right)\mathfrak{b}_{kj}\tilde{u}_{x_{m}x_{j}}\, \phi\, \rd x
\end{split}
\end{align}
where to get the last line in (\ref{IandII}) we have used
\begin{align}
 \label{IandIIcomp}
0=\int\limits_{\Omega}\sum\limits_{m,k,j=1}^{n}\mathfrak{b}_{kj}\tilde{u}_{x_{m}x_{j}}\frac{\urpartial }{\urpartial x_{k}}\left(\frac{f_{\eta_{m}}(\nabla \tilde{u})}{f(\nabla \tilde{u})}\phi\right)\rd x.
\end{align}
(\ref{IandIIcomp}) is a consequence of (\ref{intbypartszlsoln}) with $m=l$ and $\phi$ replaced by $\frac{f_{\eta_{m}}(\nabla \tilde{u})}{f(\nabla \tilde{u})}\phi$ as well as the fact that
\[
\frac{f_{\eta_{m}}(\nabla \tilde{u})}{f(\nabla \tilde{u})}\in W^{1,2}_{\mbox{\tiny{loc}}}(\Om).
\]
From (\ref{IandII}) we have
\begin{align}
 \label{IandIIimp}
\begin{split}
\int\limits_{\Omega}\sum\limits_{k,j=1}^{n}\mathfrak{b}_{kj}\tilde{v}_{x_{j}} \phi_{x_{k}}\, \rd x&=-\int\limits_{\Omega}\sum\limits_{m,k,j=1}^{n}\frac{\urpartial }{\urpartial x_{k}}\left(\frac{f_{\eta_{m}}(\nabla \tilde{u})}{f(\nabla \tilde{u})}\right)\mathfrak{b}_{kj}\tilde{u}_{x_{m}x_{j}}\, \phi\, \rd x \\
&=-\int\limits_{\Omega}(I'+I'')\phi\rd x
\end{split}
\end{align}
where (after taking the $x_{k}$ derivative of the term)
\begin{align}
\label{I'andI''}
\begin{array}{l}
I'=\sum\limits_{m,j,k,l=1}^{n}\frac{1}{f(\nabla \tilde{u})}\, \mathfrak{b}_{ml}\mathfrak{b}_{kj}\tilde{u}_{x_{l}x_{k}}\tilde{u}_{x_{m}x_{j}}, \\
I''=-\frac{1}{f^{2}(\nabla \tilde{u})}\,\sum\limits_{m,j,k,l=1}^{n}\mathfrak{b}_{kj}f_{\eta_{m}}(\nabla \tilde{u})f_{\eta_{l}}(\nabla \tilde{u})\tilde{u}_{x_{l}x_{k}}\tilde{u}_{x_{m}x_{j}}.
\end{array}
\end{align}
To simplify computation in \eqref{I'andI''} we use matrix notation. If $f=f(\nabla \tilde{u})$, $f_{\eta_{k}}(\nabla \tilde{u})=\mathfrak{b}_{k}$, $1\leq k \leq n$, then we first observe by reordering the terms in \eqref{I'andI''} that 
\begin{align*}
(I'+I'') f=  \sum\limits_{m,j,k,l=1}^{n} [ \mathfrak{b}_{nl} \tilde{u}_{x_l x_k} \mathfrak{b}_{kj} \tilde{u}_{x_j x_m} - \frac{1}{f} \, \mathfrak{b}_{l} \tilde{u}_{x_l x_k} \mathfrak{b}_{kj}\tilde{u}_{x_j x_m} \mathfrak{b}_{m} ] .
\end{align*}
Let $A=(\tilde{u}_{x_{i}x_{j}})$ and $B=(\mathfrak{b}_{ij})$, then for almost every $ x \in \Om, $
\begin{align}
 \label{add1}
(I'+I'')f=\mbox{tr}\, \, (BA)^2 -\frac{1}{f}\frac{1}{(p-1)^2} \nabla \tilde{u} BABAB (\nabla \tilde{u})^t
\end{align}
where we have used (\ref{fisdegreep}) to replace $\mathfrak{b}_{l}$. We look at
$$
\frac{ \zeta \, BABAB \, \zeta^t }{\zeta \, B \,\zeta^t} = \frac{\text{tr}( \zeta \, BABAB \, \zeta^t)}{ \text{tr}( \zeta \, B \, \zeta^t)}.
$$ 
Observe from \eqref{prooff} that $B$ is positive definite symmetric, $A$ is symmetric, and from \eqref{flaplace} that $\mbox{tr}(AB)=\mbox{tr}(BA)=0$. Using these facts we see there exists  $\m{S}$ an orthogonal matrix so that $\m{S}^t B \m{S} = B_d$ is diagonal. Let $B'_d= B_d^{1/2}$ be the obvious square root of each component of $B_d$ so that $B'_d \, B'_d = B_d$. With $A_1 = \m{S}^t A \m{S}$, it follows that
\begin{align*}
\frac{ \zeta \, BABAB \, \zeta^t }{\zeta \, B \,\zeta^t} &=\frac{\zeta \, \m{S} \m{S}^t B\m{S} \m{S}^tA\m{S} \m{S}^tB\m{S} \m{S}^tA\m{S} \m{S}^tB \m{S} \m{S}^t\, \zeta^t }{\zeta \,\m{S} \m{S}^t B \,\m{S} \m{S}^t \zeta^t} \\
&= \frac{ \zeta \m{S} \, B_d A_1 B_d A_1 B_d \, \m{S}^t \zeta^t }{ \zeta \m{S} \, B_d \, \m{S}^t \zeta^t}.
\end{align*}
If  $\xi = \zeta \m{S} \neq 0,$ then
$$
\frac{ \zeta \, BABAB \, \zeta^t }{\zeta \, B \,\zeta^t} =
\frac{\xi \, B_d A_1 B_d A_1 B_d \, \xi^t }{ \xi \, B_d \, \xi^t} =
\frac{\xi B'_d\, B'_d A_1 B'_d B'_d A_1 B'_d \, B'_d \xi^t }{\xi B'_d \, B'_d \xi^t}.
$$
Set $y = \xi B'_d \neq 0$, $E = B'_d A_1 B'_d$, and note that $E$ is symmetric
as $B'_d$, $A_1=\m{S}^t A \m{S}$, and $A$ are symmetric;
$$
\frac{ \zeta \, BABAB \, \zeta^t }{\zeta \, B \,\zeta^t} =
\frac{y \, B'_d A_1 B'_d B'_d A_1 B'_d \, y^t }{ y \, y^t}=
\frac{y \, E E \, y^t}{ y \, y^t}.
$$
Now one can easily prove the following properties of trace;
\begin{align*}
&(i)\, \, \text{tr}{(FGH)}=F_{ij}G_{jk}H_{ki} = H_{ki}F_{ij}G_{jk} =\text{tr}(HFG),\\
&(ii)\, \, \text{tr}(P^{-1} G P) = \text{tr}(G)
\end{align*}
whenever $F,G,H$ are matrices. Here (ii) follows easily from property (i) whenever $P$ is an orthogonal matrix. From these properties (i)-(ii) we have
\begin{align*}
\text{tr}(E) = \text{tr}(B'_d A_1 B'_d) = \text{tr}(B'_dB'_d A_1 ) = \text{tr}(B_d A_1)
=\text{tr}( \m{S}^t B \m{S} \m{S}^t A \m{S}) = \text{tr}(BA).
\end{align*}
Therefore, we have $\text{tr}(E)=\text{tr}(BA)=\text{tr}(AB) =0$. Similarly,
$$
\text{tr}(E^2) = \text{tr}{( (AB)^2 )}
$$
Now diagonalize $E$ using another orthogonal matrix $\m{S}_{1}$, so that $\m{S}_{1}^t E \m{S}_{1} = E_d$ with the $ij$th entries given by
$(E_d)_{ij} = e_i \delta_{ij}.$ Then
\begin{align}
\label{trofeiszero}
\begin{split}
&\text{tr}(E)= \text{tr}(E_d) = \sum\limits_{i=1}^{n} e_i =0,\\
&\text{tr}(E^2) = \text{tr}( \m{S}_{1}^t E \m{S}_{1} \m{S}_{1}^t E \m{S}_{1}) = \text{tr}(E_d^2) = \sum\limits_{i=1}^{n} e_i^2.
\end{split}
\end{align}
Moreover,
$$
\frac{ \zeta \, BABAB \, \zeta^t }{\zeta \, B \,\zeta^t} = \frac{y \m{S}_{1} \m{S}_{1}^t E\m{S}_{1} \m{S}_{1}^t E \m{S}_{1} \m{S}_{1}^ty^t}{ y\m{S}_{1} \m{S}_{1}^ty^t}
$$
so that with $z = y\m{S}_{1}\neq 0$ we also have
\begin{align}
\label{BABAB}
\frac{ \zeta \, BABAB \, \zeta^t }{\zeta \, B \,\zeta^t} = \frac{z E_d E_d z^t}{ z z^t} =
\frac{\sum\limits_{i=1}^{n} e_i^2 z_i^2} {\sum\limits_{i=1}^{n} z_i^2}.
\end{align}
Let $\kappa=z/|z| $ so that $\kappa$ is a unit vector, then \eqref{BABAB} implies
\begin{align}
\label{0BABAB}
0\leq \frac{ \zeta \, BABAB \, \zeta^t }{\zeta \, B \,\zeta^t} =\sum_{i=1}^{n} e_i^2 \kappa_i^2.
\end{align}
Without loss of generality assume that $e_1^2 $ is the largest of the $e_k^2$ then considering all possible unit vectors $\kappa$ in \eqref{0BABAB} we see that
\begin{align}
\label{e12}
\frac{ \zeta \, BABAB \, \zeta^t }{\zeta \, B \,\zeta^t} \leq \sup_{|\kappa|=1} \sum\limits_{i=1}^{n} e_i^2 \kappa_i^2 =e_1^2 .
\end{align}
Combining \eqref{add1}, \eqref{trofeiszero}, \eqref{0BABAB}, and \eqref{e12} we have
\begin{align}
\label{I''II''f}
(I'+I'')f = \text{tr}( BA)^2 - \frac{p}{p-1} \frac{ \zeta \, BABAB \, \zeta^t }{\zeta \, B \,\zeta^t}
\geq \sum\limits_{i=1}^{n} e_i^2 - \frac{p}{p-1} e_1^2.
\end{align}
Now we can use (\ref{trofeiszero}) to get
\begin{align}
\label{e1e12}
e_1 = -(\sum\limits_{i=2}^{n} e_i)\; \; \mbox{and}\; \; e_1^2 = (\sum\limits_{i=2}^{n} e_i)^2 \leq (n-1) \sum\limits_{i=2}^{n} e_i^2.
\end{align}
Using (\ref{e1e12}) in (\ref{I''II''f}) we have
\begin{align}
\label{final-1}
(I'+I'')f \geq e_1^2 +\frac{1}{n-1}e_1^2 - \frac{p}{p-1} e_1^2 = e_1^2 \left( \frac{n}{n-1} - \frac{p}{p-1}\right)
\end{align}
Finally $\frac{t}{t-1}$ is decreasing on $t > 1$ so that for $p\geq n$ we see that $(I'+I'')f \geq 0$. Combining (\ref{IandIIimp}) and (\ref{final-1}) we deduce that
\begin{align}
 \label{finalstage}
\begin{split}
\int\limits_{\Omega}\sum\limits_{k,j=1}^{n}\mathfrak{b}_{kj}\tilde{v}_{x_{j}} \phi_{x_{k}}\, \rd x &=-\int\limits_{\Omega}(I'+I'')\phi\rd x \\
& \leq - \left( \frac{n}{n-1} - \frac{p}{p-1}\right)\int\limits_{\Omega}\frac{e_1^2 (x) }{f(\nabla \tilde{u}(x)) } \phi\rd x\\
&\leq 0.
\end{split}
\end{align}
whenever $\phi\in C_{0}^{\infty}(\Omega)$ and non-negative. It follows from (\ref{finalstage}) that $\tilde{L}\tilde{v}\geq 0$ weakly in $ \Om $ when $p\geq n$. 

Let $ \de_{jk} $ denote the Kronecker delta in  the following lemma.  
\begin{lemma}
\label{v'issolnlemma}
Let $\tilde{O}, w, r, \tilde{u}, f$ be as in Lemma \ref{fislegthanu}. Let $-\infty<\he\leq- 1$. Let $\tilde{L}$ be defined as in (\ref{Lzeta0}) and $\tilde{v}=\log f(\nabla \tilde{u})$ when $ x\in\tilde{O} \cap B(w,4r)$ and $\nabla \tilde{u}(x) \neq 0$.
Let $f_{\eta_{j}\eta_{k}}= \de_{jk} $ when $\nabla \tilde{u}(x)=0$ for $1\leq j,k\leq n$. If $v'= \max \{\log f(\nabla \tilde{u}), \he \}$ then $\zeta=v'$ is locally a weak sub solution to $\tilde{L}\zeta=0$ in $ \tilde{O}\cap B(w, 4r)$.
\end{lemma}
\begin{proof} 
From Lemma \ref{localholderfornablau} we see that $v'$ is locally in $ W^{1,2} ( \tilde{O}\cap B ( w, 4r ) ). $ Given $ \varepsilon_{1}, \varepsilon_{2}, \varepsilon_{3}>0$, small, define
\[
g(x):= (\max\{v'(x)-\he-\varepsilon_{1},0\}+\varepsilon_{2})^{\varepsilon_{3}}-\varepsilon_{2}^{\varepsilon_{3}},\, \, x\in\tilde{O}\cap B(w,4r).
\]
It follows from (\ref{prooff}) and $\tilde{L}v'\geq 0$ weakly at $x\in\tilde{O}\cap B(w, 4r)$ when $v'(x)\neq\he$ (almost everywhere), that 
\begin{align}
\label{v'issoln}
\begin{split}
0 &\leq -\sum_{j,k=1}^{n}\, \, \int\limits_{\tilde{O}\cap B(w,4r)} f_{\eta_{j}\eta_{k}}(\nabla \tilde{u})(\phi g)_{x_j} v'_{x_k} \rd x \\
&\leq -\sum_{j,k=1}^n \, \, \int\limits_{\tilde{O}\cap B(w, 4r)} g f_{\eta_j\eta_{k}}(\nabla \tilde{u}) \phi_{x_j} v'_{x_k} \rd x.
\end{split}
\end{align}
whenever $\phi \in C_0^\infty ( \tilde{O} \cap B ( w, 4r ) )$ and non-negative. Using (\ref{v'issoln}), the bounded convergence theorem, and letting first $\varepsilon_{1}\to 0$, then $\varepsilon_{2}\to 0$, and finally $\varepsilon_{3} \to 0$, we get Lemma \ref{v'issolnlemma} as desired.
\end{proof}
\section{Advanced Regularity Results}
\label{advancedregres}
In this section we begin the proof of Theorem \ref{alv13.1} by proving three lemmas. To this end, let $O, f, u, \hat z, \rho, \mu_{f}, p, n$ be as in Theorem \ref{alv13.1}.
\begin{lemma}
\label{prop}
There exists a constant $c=c(p, n, c_*)$ and a set $Q\subset \partial O\cap B(\hat z, \rho)$ such that
\[
\mu_{f}((\partial O\cap B(\hat z, \rho))\setminus Q)=0.
\]
Moreover, for every $w\in Q$ there exists arbitrarily small $r=r(w)$, $0<r\leq 10^{-10}$, such that
\[
 \overline{B}(w, 100 r ) \subset B ( \hat z, \rho ) \; \mbox{and}\; \mu_{f} ( B (w, 100 r ) ) \leq c\, \mu_{f}( B(w, r)).
\]
\end{lemma}
\begin{proof}
It follows from Lemma \ref{uismu} that $\mu_{f}(B(x, t))> 0$ whenever $x\in \partial O$ and $\partial O\cap B(x,t)\subset \partial O \cap B(\hat z, \rho)$. We show for $ c > 0 $ large enough that $\mu_{f}(\Theta)=0$ where
\[
\Theta := \left\{x \in \partial O \cap B (\hat z, \rho):\;  \liminf\limits_{t \to 0 }\, \,  \frac{\mu_{f} ( B ( x,100 t) ) }{\mu_{f} ( B ( x, t ) )} \geq c \right\}.
\]
Then the desired set $Q$ in Lemma \ref{prop} will be the complement of $\Theta$, i.e, $Q=(\partial O \cap B (\hat z, \rho))\setminus \Theta$. To show that $\mu_{f}(\Theta)=0$, we first see from the definition of $\Theta$ that for every $x\in\Theta$ there exists $t_{0}=t_{0}(x)$ with
\begin{align}
\label{iterate}
(c/2)\mu_{f}(B(x,t)) \leq \mu_{f}(B(x, 100t))\;\mbox{for every}\; t\in (0, t_{0}).
\end{align}
Then iterating (\ref{iterate}) we obtain
\[
\lim\limits_{t\to 0}\frac{\mu_{f}(B(x,t))}{t^{n+1}}=0 \; \mbox{whenever}\; x\in\Theta
\]
provided $c$ in (\ref{iterate}) is large enough. It follows that $ \mu_f |_{\Theta}  $ is absolutely continuous with respect to $ \rh^{n+1} $ measure. Since ${\rh}^{n+1} ( \rn{n})=0$ we conclude from our earlier remark that Lemma \ref{prop} is true.
\end{proof}
Next using translation and dilation invariance of (\ref{flaplace}), we work in a different domain. To this end, let
\[
w \in Q\subset\partial O\cap B(\hat z, \rho)
\]
be fixed and let $r=r(w)$ be a corresponding radius as in Lemma \ref{prop}. We first set
\[
u'(x):=\frac{u(w+rx)}{\esssup\limits_{B(w,10r)} u}\; \mbox{when}\; w+rx\in B(\hat z,\rho)
\]
and define
\[
\Omega':=\{x:\; w+rx\in O \cap B(\hat z,\rho)\}.
\]
We observe that $u'$ is a weak solution to (\ref{flaplace}) in $\Omega'$ as (\ref{flaplace}) is invariant under translation and dilation. Moreover, $u'>0$ is continuous in $B(\zeta, \rho/r)$ with $u'\equiv 0$ on $B(\zeta, \rho/r)\setminus\Omega'$ provided that $\zeta=(\hat z-w)/r$. As in (\ref{ast}), there exists a finite Borel measure $\mu'_{f}$ on $\mathbb{R}^{n}$ with support in $\partial\Omega'\cap \overline{B}(\zeta, \rho/r)$ associated with $u'$.

We also note that
\[
\mu'_{f}(E)=\frac{r^{p-n}}{\left(\esssup\limits_{B(w,10r)} u\right)^{p-1}}\, \mu_{f}(\Xi(E))
\]
whenever $E$ is a Borel set and $\Xi(E):=\{w+rx:\; x\in E\}$.

As \eqref{flaplace} is invariant under translation and dilation without loss of generality we can assume that $w=0$, $r=1$ with $B(0,100)\subset B(\hat{z},\rho)$. From Lemmas \ref{uismu} and \ref{prop}, we obtain for some $c=c(p,n,c_*)\geq 1$ and $2\leq t\leq 50$ that
\begin{align}
\label{u'ismu}
c^{-1} \leq \mu'_{f} ( B (0, 1) ) \leq \esssup\limits_{B(0, 2 )}u' \leq \esssup\limits_{B (0, t )} u' \leq  c\, \mu'_{f} ( B (0, 100)) \leq c^2 .
\end{align}

By definition of $u'$ and H\"older continuity of $u$ near $\partial O$, it is easily seen that there exists some $\tilde{z}\in \partial B(0,10)$ with $u'(\tilde{z})=1$, and
\begin{align}
\label{d1c}
c^{-1}_{-}\leq d(\tilde{z}, \partial\Omega')\; \mbox{for some}\; c_{-} = c_{-}(p,n,c_*) \geq 1.
\end{align}

Let $M$ be a large number where we allow $M$ to vary but shall fix it to satisfy several conditions after (\ref{3.35}). After that we choose $ s = s(M)>0$ sufficiently small with $0<s<< e^{-M}$.  Let $\delta, \delta'$ be given such that $0<\delta'<\min (\delta, 10^{-5})$ and choose $M>0$ so large that
  \begin{align}
       \label{Mandmeasure}
       \mbox{if}\; \mu'_{f} (B(z, t))= M t^{n-1}\; \mbox{for some}\; t = t (z)\leq 1\; \mbox{then}\; t \leq \delta'
       \end{align}
       where $z\in\partial\Omega'\cap\overline{B}(0, 15)$. Existence of such $M=M(\delta')\geq 1$ follows from \eqref{u'ismu}. Following \cite{W93}, we observe from (\ref{Mandmeasure}) for each $z\in\partial\Omega'\cap \overline{B}(0,15)$ that there exists a largest $t$ with $s \leq t \leq 1$ such that either
\begin{align} 
\label{stoptime} 
\begin{array}{l}
(a) \;\mu'_f ( B(z, t) ) = M t^{n-1}\;, t > s  \\
\mbox{ or } \\
(b) \; t = s. 
\end{array}
\end{align}

  Using the Besicovitch covering theorem (see \cite{Ma95}) we now obtain a covering $\{ B ( z_k, t_k ) \}_{k=1}^N $
   of  $ \partial \Omega'\cap \overline{B}(0, 15)$, where $ t_k $ satisfies either (a) or (b) in \eqref{stoptime}.
Then each point of $\bigcup_{k=1}^N B ( z_k, t_k ) $  lies in at most $ c = c (n) $ of $\{ B ( z_k, t_k ) \}_{k=1}^N $.
Let $\m{G} = \m{G}
 _M$ and $\m{B}= \m{B}_M$ be the set of all balls in this covering for which (a) and (b) in \eqref{stoptime} hold respectively.

Let $ c_{-} $ and $\ti z, $ be as in  (\ref{d1c}) and set $ r_1 = (8 c_{-})^{-1}$. Choosing $ \delta' $ smaller (so $M$ larger) if necessary we may assume, thanks to (\ref{Mandmeasure}), that
     \begin{align}
     \label{unionofballs}
     \bigcup_{k=1}^N \overline{B} ( z_k, 6 t_k )\cap B ( \ti z, 6r_1 ) = \es.
     \end{align}
      Also put
      \[
      \Omega'' = \Omega' \cap B (0, 15 ) \setminus \bigcup_{k=1}^N \overline{B}( z_k, t_k )\; \mbox{and}\; D=\Omega'' \sem \overline{B}(\ti z, 2r_1).
       \]
Let $u''$ be a positive weak solution to (\ref{flaplace}) in $D$ with continuous boundary values,
   \begin{displaymath}
    u''(x)\equiv\left\{
    \begin{array}{cl}
    0&\; \mbox{when }\; x\in \ar \Omega''\\
    \essinf\limits_{\overline{B}(\ti z, 2r_1)} u' &\; \mbox{when} \; x\in\ar B ( \ti z, 2 r_1 ).
    \end{array}
    \right.
   \end{displaymath}
We extend $ u''$ continuously to $\overline{B}(0, 15)$ (also denoted $u''$) by putting
\begin{displaymath}
 u''(x)\equiv\left\{
 \begin{array}{cl}
 0&\; \mbox{when } \; x\in\overline{B}(0, 15) \setminus \Omega''\\
    \essinf\limits_{\overline{B}(\ti z, 2r_1)} u' &\; \mbox{when} \; x\in\overline{B} ( \ti z, 2 r_1 ).
 \end{array}
 \right.
\end{displaymath}
We note that $ u'' \leq u' $ on $ \partial D $ so by the maximum principle for weak solutions to (\ref{flaplace}) we have
 $ u''\leq u' $ in $D$.  Also, $\partial D$ is locally $(n,r'_0)$ uniformly fat where $ r_0' $
depends only on $n$ and $ r_0 $ in Theorem \ref{alv13.1} when $p=n$. Next we prove  
 \begin{lemma}
 \label{derofuisbdd}
For all $x\in D$ we have $ |\nabla u''|\leq c M^{\frac{1}{p-1}}$ where $c=c(p, n, c_*)$.
 \end{lemma}
 \begin{proof}
Let $x\in D$, and choose $y\in \partial D$ such that $|x-y|=d(x,\partial D)=d$. We first prove Lemma
\ref{derofuisbdd} when $y\in \partial B(z_{k}, t_{k})$ and $x\in B(z_{k}, 2t_{k})$. The same reasoning can be applied when
$y\in \partial B(0,15)$ or $y\in\partial B(z,2r_{1})$. To this end, let $\epsilon>0$ be given and
set
\[
f^{\epsilon}(\eta):=\int\limits_{\mathbb{R}^{n}}f(x)\psi_{\epsilon}(\eta-x) \, \dr{d}x
\]
where $\psi\in C^{\infty}_{0}( B (0, 1) )$ with
\[
\int\limits_{\mathbb{R}^{n}}\psi \dr{d}x=1\; \mbox{and}\; \psi_{\epsilon}(x)=\frac{1}{\epsilon^{n}}\psi\left(\frac{x}{\epsilon}\right)\; \mbox{whenever}\; x\in\rn{n}.
\] 
We note that $ f^\ep $ is no longer homogeneous but $ f^\ep $ is infinitely differentiable. Moreover, whenever $ \eta, \xi \in \rn{n}$ we have
\begin{align} 
\label{reg}
  c^{-1} ( \ep + |\eta| )^{p-2} \, |\xi |^2 \, \leq \, \sum_{j,k=1}^n  \frac{ \ar^2 f^{\ep}}{ \ar \eta_{j} \eta_{k} } (\eta ) \xi_j \xi_k \leq   \, c ( \ep + |\eta| )^{p-2} |\xi |^2
\end{align} 
where $ c = c ( p, n, c_*) \geq 1$. Let $u''_{\epsilon}$ be a weak solution to \eqref{flaplace}  in $ D $  with $ f $ replaced by $ f^\ep $  and the same continuous boundary values as $u''.$
 Then \eqref{ast} holds with $ f, u $ replaced by $ f^\ep, u''_\ep. $  Using \eqref{reg}, an analogue of Lemma \ref{localholderfornablau}, and Schauder type estimates we see that $ u''_\ep $ is infinitely differentiable in $ \Om'' $ and that $ \ze = u^{''}_\ep $ is a pointwise solution to $ L^{\star} \ze = 0$ where
\begin{align}
\label{Lstar}
L^{\star}\zeta:=\frac{1}{(\epsilon+|\nabla u''_{\epsilon}|)^{p-2}}\sum\limits_{j,k=1}^{n} f^{\epsilon}_{\eta_{j}\eta_{k}}(\nabla u''_{\epsilon}) \zeta_{x_{j}x_{k}}.
\end{align}
Moreover, if we let
\[
\tilde{\ph}(w)=\frac{e^{-\m{N}|w-z_{k}|^{2}}-e^{-4\m{N}t^{2}_{k} }}{e^{-\m{N}t^{2}_{k}}-e^{-4\m{N}t^{2}_{k}}}.
\]
Then $L^{\star}\tilde{\ph}\geq 0$ in $B(z_{k}, 2t_{k})\setminus B(z_{k}, t_{k})$ if $\m{N}=\m{N}(p,n,c_*)$ is sufficiently large.
Thus if  \[ \Ph (w)=(\esssup\limits_{B(z_{k}, 2t_{k})} u'\, )(1-\tilde{\ph}(w)) \]
then $L^{\star}\Ph \leq 0$ in $B(z_{k}, 2t_{k})\setminus B(z_{k}, t_{k})$ .
 Using this fact,
the maximum principle for solutions to \eqref{Lstar}, $ u''  \leq u', $ and comparing boundary values, we conclude that
$ u''_\ep  \leq \Ph $  in $B(z_{k}, 2t_{k})\setminus \bar B(z_{k}, t_{k}).$ 
 Letting $\epsilon\to 0$, we deduce from the usual variational type arguments and an analogue of  Lemma
\ref{localholderfornablau} for $u_\ep'' $  that subsequences of $\{ u''_{\epsilon}\}, \{ \nabla u''_{\epsilon}\}$ converge pointwise to $u'',\nabla u''$ in $D$ and uniformly on compact subsets of $D.$ Hence 
\begin{align}
\label{uisleqPh}
u'' \leq \Ph\;  \mbox{in}\; B ( z_k, 2 t_k ) \sem \bar{B} ( z_k, t_k ).
\end{align}
Using \eqref{uisleqPh} and applying Lemma \ref{localholderfornablau} to $u''$ we see that
     \begin{align}
     \label{nablau'epsilon}
     | \nabla u'' (x)  |  \leq \frac{c}{d}\; u''(x) \leq \frac{c}{d}\; \Ph (x) \leq  \frac{c^2}{t_k} \, \esssup\limits_{B(z_k, 2t_k)} u'.
     \end{align}
     where $d=d(x,D)$. Lemma \ref{uismu} and (\ref{Mandmeasure})-(\ref{unionofballs}) imply
\begin{align}
\label{tk1-p}
t_k^{1-p}  \esssup\limits_{B(z_k, 2 t_k)} (u')^{p-1} \leq c\, t_k^{1-n}\, \mu'_{f} ( B (z_k, 4 t_k) ) \leq c^2  M.
\end{align}
Combining (\ref{nablau'epsilon}) and (\ref{tk1-p}) we see that Lemma \ref{derofuisbdd} holds for $u''$  at points in $ D $  which are also in   $ \bigcup  B ( z_k, 2 t_k) \sem \bar B ( z_k, t_k) . $ Similar arguments also give this inequality at points near $ \ar B (0, 15)$ and  $\ar B (  \ti z, 2 r_1)$. Thus there exists an open set $W$ with $\ar D \subset W$ and  $ |\nabla u'' | \leq c  M^{1/(p-1) } $ in $ W \cap D $
where $ c = c (p,n,c_*). $   Applying Lemma \ref{v'issolnlemma} to $u''$, then a maximum principle for weak subsolutions to $\tilde{L}$ defined as in (\ref{Lzeta0}), we see that Lemma \ref{derofuisbdd} holds for every $x\in D$.
 \end{proof}
The  proof of the  next lemma is essentially the same as in \cite[Lemma 8]{ALV13}. For  completeness we give the arguments here.
\begin{lemma}
\label{secderint}
The functions $| \nabla u'' |^{p-2}\, |u''_{x_j x_k}|$ for $1\leq j,k\leq n$ are all integrable in $D$. 
\end{lemma}
\begin{proof}
 Let $ \Lambda \subset \partial \Omega''$ be the set of points where $ \partial \Omega'' $ is not smooth.
     Clearly $ \rh^{n-1} (\La) = 0 .$
    If  $ \hat x \in \partial D \setminus\Lambda$, then  $ \hat x $ lies in exactly one of
        the finite number of spheres which contain points of $ \partial D. $ Let $ d' ( \hat x ) $ denote the distance from
    $ \hat x $ to the union of spheres not containing $ \hat x $
    but containing points of $ \partial D. $ If $ d' = d' ( \hat x ) < s/100, $ then from Lemma \ref{ballminusball} applied to
    $ u'' $ we see that each component of
     $ \nabla u'' $ has a H\"older continuous extension to $ B ( \hat x, 3 d'/4 ). $
     Also from H\"older continuity, Lemmas \ref{ballminusball} and \ref{derofuisbdd} we see that
   \begin{align}
   \label{3.10}
   \begin{split}
 \frac{1}{c} \sum\limits_{j, k =1}^n\, \, \int\limits_{ D \cap B ( \hat x, \frac{d'}{8} )}& \, | \nabla u'' |^{p-2} \, | u''_{x_j x_k} | \, \dr{d}x\\
&      \leq (d')^{\frac{n}{2}}  M^{\frac{p-2}{2(p-1)}} \sum\limits_{j, k =1}^n\, \, \left(\, \, \int\limits_{D \cap B ( \hat x, \frac{d'}{8} )} \, |\nabla u''|^{p-2} \, | u''_{x_j x_k} |^2 \, \dr{d}x \right)^{\frac{1}{2}}\\
     & \leq  c (d')^{\frac{(n - 2)}{2}} \,  M^{\frac{p-2}{2(p-1)}}\left(\, \, \int\limits_{ D \cap B ( \hat x, \frac{d'}{2} )} \, | \nabla u''|^{p} \, \dr{d}x \right)^{\frac{1}{2}}
     \\
     & \leq \, c^2 M \, (d')^{(n-1)}.
       \end{split}
       \end{align}
       To prove Lemma \ref{secderint} we assume as we may that $ B ( z_l, t_l ) \not \subset B ( z_\nu, t_\nu ) $ when $ \nu \neq l, $ since otherwise we discard one of these balls.  Also from  a well known covering theorem we get a covering $ \{ B (y_i, \frac{1}{20} d' (y_i) ) \} $ of $ \ar D \sem \La $ with the property that
  $ \{ B (y_i, \frac{1}{100} d' (y_i) ) \}$ are pairwise disjoint.  From \eqref{3.10} we find that
     \begin{align}
     \label{3.11}
     \begin{split}
       \sum\limits_{i,j,k}\, \, \int\limits_{D \cap B ( y_i, \frac{1}{8} d'(y_i) ) }| \nabla u'' |^{p-2} |u''_{x_j x_k}|\, \dr{d}x \, & \leq c \,  M \sum\limits_{i} (d' (y_i))^{n-1} \\
      & \leq c^2 M\, \rh^{n-1} ( \ar D ).
			\end{split}
     \end{align}

Let $ d ( x ) $ denote $ d ( x, \partial D) $. We choose a covering $ \{ B ( x_m, \frac{1}{2} d ( x_m )\} $ of $ D $ with
  $ \{ B ( x_m, \frac{1}{20} d ( x_m )\}, $ pairwise disjoint.
  We note that if $ x \in D $ and $ y \in \ar D $ with $ | y - x | = d ( x ), $ then $ y \in \ar D \sem \La. $
Indeed otherwise $ y $ would be on the boundary of at least two balls contained in the complement of $ D $ and so
by the no containment assumption above, would have to intersect $ B ( x, d (x) ),$ which clearly is a contradiction.
Also we assert that if $ d ( x ) \leq 1000 s, $ then $ d ( x ) \leq  \kappa \, d' (y) $ where $ \kappa $ can depend
on various quantities including the configuration of the balls,  $ \{ B(z_k,t_k)\} $ but is independent of
$ x \in D $ with $ d ( x ) \leq 1000 s. $  Indeed from the no containment assumption one deduces that otherwise there exists
sequences $ ( x_m), (y_m), (y'_m), $ with $ x_m \in D,  y_m \in C_1, y'_m \in C_2, $ where  $ C_1, C_2 $ are spheres in
 $ \{ \ar B ( z_j, r_j ) \}_1^N $ with $ C_1 \not =  C_2 $ and
\begin{align}
 \label{condic}
 \begin{split}
  &  | x_m - y_m | = d ( x_m), |y_m - y_m' | = d'(y_m) \, \, \mbox{and}\\
  & \mbox{as }
 m  \rar \infty,   d(x_m)/ d'( y_m ) \rar \infty, \mbox{ with }  x_m, y_m, y'_m \rar w
 \in C_1 \cap C_2 \subset \La .  
\end{split} 
\end{align}
From basic geometry we see that either
 $ (i) \, C_1 \cap C_2 = w. $ or
 $ (ii) \,  C_1 \cap C_2 $ is an
 $ n - 2 $ dimensional sphere.  If  $  (i) $ holds then $ C_1, C_2 $ are tangent, so clearly for large $ m, \, d ( x_m) \leq c d '( y_m) . $  If  $ (ii)$ holds then   considering the tangent planes to  $ C_1, C_2 $  through $ w $   we see for large $m$
that  \[  d ( x_ m ) \leq c  d ( x_m, C_1 \cap C_2 )   \leq  c^2 d' ( y_m )  \] where $ c $ is independent of $ m. $ In either case we have reached a contradiction to \eqref{condic}.
Hence our assertion is true.

From this analysis and our choice of covering of $D$ we see that for a given $ B ( x_m, \frac{1}{2}d ( x_m ) ) $ with $ d ( x_m ) < 1000 s, $
there exists $ j = j( m )$  with $ B ( x_m,\frac{1}{2} d ( x_m ) ) \subset B ( y_j , \kappa' d' (y_j) ) $ for some $ 0< \kappa' < \infty $ independent of $ m. $

Let $ S_l, l = 1, 2, 3, $ be disjoint sets of integers defined as follows.
\begin{displaymath}
 \left\{
 \begin{array}{ll}
 m \in S_1 & \mbox{if}\; d ( x_m ) \geq 1000 s,\\
 m \in S_2 & \mbox{if}\; m \not \in S_1\; \mbox{and}\; \not\exists\; j \; \mbox{with}\; B ( x_m, \frac{1}{2} d (x_m)) \subset B ( y_j, \frac{1}{8} d' (y_j )),\\
 m\in S_{3}& \mbox{if} \; m\; \mbox{is not in either}\;  S_1\; \mbox{or}\; S_2.
 \end{array}
 \right.
\end{displaymath}
    Let
    \[
    K_l =  \sum\limits_{m\in S_l}\, \, \int\limits_{ D \cap B ( x_m, \frac{1}{2} d(x_m) ) } | \nabla u''|^{p-2} |u''_{x_j x_k}| \dr{d}x \; \mbox{for}\; l =1, 2, 3.
    \]
        Then
         \begin{align}
         \label{3.12}
         \int\limits_{D}| \nabla u''|^{p-2} |u''_{x_j x_k}| \dr{d}x \, \leq \, K_1 \, + \, K_2 \, + \, K_3 .
          \end{align}
From Lemma \ref{localholderfornablau} and the same argument as in (\ref{3.10}) we see that
  \begin{align}
  \label{3.13}
  K_1 \leq c\, M \sum_{m\in S_1} d(x_m)^{n-1}
    \leq c^2 M s^{-1}
    \end{align}
where we have used disjointness of our covering $\{B(x_m, \frac{1}{20} d(x_m))\}$. Using disjointness of these balls and (\ref{3.11}) we get
\begin{align*}
K_3 \leq c\, M \rh^{n-1} (\ar D).
\end{align*}
Finally if $m\in S_2$ then as discussed earlier there exists $j=j(m)$ with $d( x_m)\approx d'(y_j)$, where proportionality constants are independent of $m$, so $B( x_m, \frac{1}{2} d ( x_m ) ) \subset B ( y_j, \kappa' d' ( y_j ) )$. From disjointness of  $ \{ B ( x_m, \frac{1}{20} d (x_m) ) \} $ and a volume type argument we deduce that each $j$ corresponds to at most $ \kappa''$ integers $m\in S_3$ where $\kappa''$ is independent of $j$. Using this fact, an argument as in (\ref{3.10}), as well as disjointness of $ \{ B ( y_i , \frac{1}{100} d' (y_i) ) \}$, we conclude that there is a $\ti \kappa$ with $0<\ti \kappa < \infty, $ satisfying
  \begin{align}
  \label{3.15}
  K_2 \, \leq \, \ti \kappa M \sum_{m \in S_2} d (x_m)^{n-1}  \leq
    \ti \kappa^2  M \sum_{j} d' (y_j)^{n-1}  \,   \leq \ti \kappa^3 M \, \rh^{n-1} ( \ar D ) .    \end{align}
    Using (\ref{3.13})-(\ref{3.15}) in (\ref{3.12}) we find that Lemma \ref{secderint} is valid.
 \end{proof}
We next show that there exists $c=c(p, n, c_*)\geq 1$ such that
 \begin{align}
 \label{finitemeasure}
 c^{-1}\leq \mu''_{f}(\partial\Omega''\cap B(0,10))\leq \mu''_{f}(\partial\Omega'')\leq c.
 \end{align}
To prove (\ref{finitemeasure}), it follows from Lemmas \ref{fislegthanu}-\ref{uismu}, (\ref{unionofballs}), and the fact $u'(\tilde{z})=1$ that $u''\geq 1/c$ on $\partial B(\tilde{z}, 4r_{1})$ for some $c=c(p,n,c_*)\geq 1$. Let $l$ denote the line from the origin through $ \tilde{z}$ and let $ \zeta_{1}$ be the point on this line segment in $ \partial B(\tilde{z}, 4r_{1})\cap B ( 0, 10)$. Let $ \ze_2 $ be the point on the line segment from $ \ze_1 $ to the origin with $ d ( \ze_2, \ar \Omega'') = \frac{1}{20} r_1 $  while $d(\ze, \ar\Om'')>\frac{1}{20} r_1$ at all other points on the line segment from $ \ze_1 $ to $ \ze_2. $ Then from (\ref{d1c}), Lemma \ref{fislegthanu}, and the above discussion we see that $ u'' ( \ze_2 ) \geq 1/c'$ for some $c'=c'(p,n,c_*)\geq 1$.
Also, $ B ( \ze_2, \frac{1}{2} r_1 ) \subset B (0, 10). $ Let $ \hat \ze $ be the point in $ \ar \Om'' $ with $ | \hat \ze - \ze_2 | = d ( \ze_2, \ar \Om'' ). $    Applying Lemma \ref{uismu} with  $w = \hat \ze,  r = 2 d ( \ze_2, \ar \Om'')$, we deduce that the left hand inequality is valid.  The right hand inequality in this claim follows once again from Lemma \ref{uismu} and $ u'' \leq u'$.

 Using Lemmas \ref{derofuisbdd}-\ref{secderint} and (\ref{finitemeasure}) we prove the following lemma.
 \begin{lemma}
 \label{logfgradu}
 There exists $ c = c (p, n, c_*) $ such that
   \[
  \int\limits_{\partial D} | \log f(\nabla u'') | \frac{f(\nabla u'')}{|\nabla u''|} \, \rd\rh^{n-1} \, \leq \, c \log M.
  \]
  \end{lemma}
  \begin{proof}
  Let
  \[
\log^+ t := \max \{ \log t, 0 \}\; \mbox{and}\; \log^- t := \log^+ (1/t)\; \mbox{for}\; t \in (0, \infty).
\]
We first give a proof of Lemma \ref{logfgradu} for $\log^{+}f(\nabla u'')$. To this end, we observe from Lemma \ref{ballminusball}  that
\begin{align}
    \label{3.16}
    \rd \mu''_{f}=  p \frac{f(\nabla u'')} {|\nabla u''|} \rd\rh^{n-1}> 0\; \mbox{on}\; \partial \Omega'' \setminus \Lambda.
    \end{align}
It follows from Lemma \ref{derofuisbdd}, (\ref{finitemeasure}), (\ref{3.16}), and $\rh^{n-1}(\Lambda)=0$ that
\begin{align} 
\label{3.18a}
  \int\limits_{\partial \Omega''} \log^{+}(f(\nabla u'')) \frac{f(\nabla u'')}{|\nabla u''|} \, \rd\rh^{n-1} \leq c \log M\, \mu''_{f}(\partial\Omega'')\leq c^{2}\log M.
\end{align} 
To prove Lemma \ref{logfgradu} for $\log^{-}f(\nabla u'')$, fix $\xi$, $-\infty<\xi<-1,
$ and set $v''(x)=\max(\log f(\nabla u''), \xi)$ when $x\in \overline{D}\setminus \Lambda$. Given small $\theta>0$ we set
\begin{align}
\label{Lahe}
  \La ( \he ) = \{ x \in D : d ( x, \La ) \leq \he \}\; \mbox{and}\; D ( \he ) = D \sem \La (\he).
\end{align}
Observe from  Lemmas \ref{derofuisbdd}-\ref{secderint} and (\ref{finitemeasure}) that
 \[
  \tilde{L} u''(x) =  \nabla \cdot \left( \m{D} f(\nabla u''(x)) \right) = 0
  \]
exists pointwise  for almost every $ x \in D ( \he ) $ and is integrable on $ D (\he). $  Put
 \begin{align}
 \label{3.19}
 \begin{split}
 I (\he) &= \int\limits_{D(\he)} \sum\limits_{j,k=1}^{n}\frac{\partial }{\partial x_{j}}\left(f_{\eta_{j}\eta_{k}} (\nabla u'' )  u''_{x_{k}}\right)\, v'' \, \dr{d}x + \int\limits_{D(\he)}  \sum_{j,k= 1}^n f_{\eta_{j}\eta_{k}} (\nabla u'' )  u''_{x_k} \, \frac{\partial v''}{\partial x_{j}} \dr{d}x \\
&= I_1 ( \he ) + I_2 ( \he )
\end{split}
 \end{align}
From \eqref{Lahe} and $p-1$ homogeneity of derivatives of $f$  we see that $ I_1 (\he ) = 0$. To handle $ I_2 ( \he ) =I(\he) $, we first use a barrier argument as in Lemma \ref{derofuisbdd}, and then we use Lemma \ref{ballminusball} to deduce that there exists some $ c = c ( p, n, c_*) \geq 1, $ such that
       \begin{align}
   \label{3.20}
   c^{-1}\leq  | \nabla u''|  \leq c\;  \mbox{on}\; \bar B ( \ti z, 2 r_2 ) \sem B ( \ti z, 2 r_1) \; \mbox{where}\; r_2 = ( 1 + c^{-1}) r_1.
   \end{align}
Let $\phi$, $0\leq \phi \leq 1$, be an infinitely differentiable function in $\mathbb{R}^{n}$ with $\phi\equiv 1$ on $\rn{n} \setminus B( \ti z, 2 r_2 )$, $|\nabla\phi|\leq c\, r_{1}^{-1}$, and $\phi\equiv 0$ on an open set containing $B(\ti z, 2 r_1 )$.   From (\ref{d1c}) and the definition of $r_{1}$ we have $|\nabla \phi|\leq c^{2}$. Rearranging $I_{2}(\theta)$  and writing    
  $ f_{\eta_{j}\eta_{k}}$  for    $     f_{\eta_{j}\eta_{k}}  ( \nabla u'' ) $  
we have
\begin{align*}
  I_2 (\he) & = \int\limits_{D(\he)}  \sum\limits_{j,k= 1}^n f_{\eta_{j}\eta_{k}} (\phi u'' )_{x_k} \, v''_{x_j} \dr{d}x + \int\limits_{D(\he)}  \sum\limits_{i,k= 1}^n f_{\eta_{j}\eta_{k}} ( (1 - \phi) u'' )_{x_k} \, v''_{x_j} \dr{d}x \\
  &  = I_{21} ( \he ) + I_{22} ( \he ).
  \end{align*}
It follows from Lemmas \ref{localholderfornablau}-\ref{ballminusball}, (\ref{3.20}), and an argument similar to (\ref{3.10}) that
  \begin{align}   
  \label{3.27a}  
 \begin{split}
|I_{22}(\theta)| &\leq \int\limits_{B(\ti z, 2 r_2 )}  \sum\limits_{j,k= 1}^n |f_{\eta_{j}\eta_{k}}|  (1 - \phi) |\nabla u''| \, |v''_{x_j}| \rd x \\
							&\hspace{.5cm} + \int\limits_{B(\ti z, 2 r_2 )\setminus B(\ti z, 2 r_1 )}  \sum\limits_{j,k= 1}^n |f_{\eta_{j}\eta_{k}}| |\nabla \phi| u'' \, |v''_{x_j}| \rd x \\
							&\leq c, 
\end{split}
\end{align}
where $c$ is independent of $\theta$. From \eqref{3.20} and Lemmas \ref{localholderfornablau},
 \ref{ballminusball}, \ref{derofuisbdd} and \ref{secderint} we see that the integrand in the integral defining  $ I_{21} ( \he ) $ is bounded by an integrable function independent of  $ \he $.  Using this fact and the Lebesgue dominated convergence theorem we find  that
\begin{align} 
\label{3.28a}
{\ds \lim_{\he \rar 0 }  I_{21} (\theta)=\int\limits_{D}  \sum\limits_{j,k= 1}^n f_{\eta_{j}\eta_{k}} (\phi u'' )_{x_k} \, v''_{x_j} \dr{d}x} =: I'_{21}. 
\end{align}
We claim that  $ I'_{21} \leq 0$. To verify this claim let  $ u^* = u^* (\de ) = \max ( u'' - \de, 0 )$. Convoluting $ \ph u^* $ with an approximate identity and taking limits we see from Lemma \ref{v'issolnlemma}
 that
 \[   
 \int\limits_{D}  \sum_{j,k= 1}^n f_{\eta_j \eta_k} ( \ph u^{*})_{x_k} \, v''_{x_j} dx \leq 0.  
 \]
Moreover, once again from Lemmas \ref{derofuisbdd} and \ref{secderint}, we observe that the above integrand is dominated by an integrable function independent of $\delta$. From this fact, the above inequality, and the Lebesgue dominated convergence theorem we get assertion $ I'_{21} \leq 0$. Using \eqref{3.27a}, \eqref{3.28a}, and above claim we conclude that
     \begin{align} 
     \label{3.28} 
     \lim\limits_{\he \rar 0} I (\he )  \leq c. 
     \end{align}
On the other hand from \cite[chapter 5]{EG92} and \eqref{3.19} we see that integration by parts can be used to get
\begin{align*}
I(\theta)&=\int\limits_{\partial D(\theta)} v''\sum\limits_{j,k=1}^{n}f_{\eta_{j}\eta_{k}}u''_{x_{k}}\mathfrak{n}^{j}\dr{d}\rh^{n-1}  =  (p-1) \int\limits_{\partial D(\theta)} v''\sum\limits_{j=1}^{n}f_{\eta_{j}} \mathfrak{n}^{j}\dr{d}\rh^{n-1}
\end{align*}
 where $\mathfrak{n}$ is the outer unit normal to $\partial D(\theta)$. From Lemma \ref{derofuisbdd}, the dominated convergence theorem, and the definition of $D(\theta)$, we have
\begin{align}
 \label{3.29}
\begin{split}
 \int\limits_{ \partial D ( \theta) \setminus \partial B ( \ti z, 2 r_1 ) }
 v'' \,\sum\limits_{j = 1}^{n} f_{\eta_{j}} \,  \mathfrak{n}^j \dr{d}\rh^{n-1}  \to  \int\limits_{ \partial \Omega'' \setminus \Lambda}  v''
 \,\sum\limits_{j}^{n} f_{\eta_{j}} \, \mathfrak{n}^j \dr{d}\rh^{n-1}\; \mbox{as}\; \theta \to 0 .
\end{split}
 \end{align}
From (\ref{fisdegreep}), (\ref{3.28}), and (\ref{3.29}) we deduce
\begin{align}
\label{3.30}
\begin{split}
 \int\limits_{\partial\Omega''\setminus \Lambda}	 v'' \sum\limits_{j=1}^{n}f_{\eta_{j}}\mathfrak{n}^{j}\dr{d}\rh^{n-1}&=-p\int\limits_{\partial\Omega''\setminus\Lambda}v''\, \frac{f(\nabla u'')}{|\nabla u''|}\, \dr{d}\rh^{n-1}\\
 &\leq p(p-1)^{-1} \lim\limits_{\theta\to 0}I(\theta)+c\leq 2c.
\end{split}
\end{align}
 where we have also used the fact that $\mathfrak{n}=-\frac{\nabla u''}{|\nabla u''|}$ and 
 \[ 
 \left|\, \, \int\limits_{\partial B(\tilde{z}, 2r_{1})}v''\sum\limits_{j=1}^{n}f_{\eta_{j}}\mathfrak{n}^{j}\dr{d}\rh^{n-1}\right| \leq c=c(p,n,c_*).
 \]  
Letting $ \xi \to - \infty $ in (\ref{3.30}) and using the monotone convergence theorem we see that (\ref{3.30}) holds with $v''$ replaced by $ \log f(\nabla u'')$.
 Finally from \eqref{3.30} for $\log f(\nabla u'')$ and \eqref{3.18a} we conclude the validity of Lemma \ref{logfgradu}.
  \end{proof}
\section{Proof of Theorem \ref{alv13.1}}
\label{mainproof}
In this section we first give a proposition which will be a  consequence of lemmas we obtained in section \ref{advancedregres} and then we prove Theorem \ref{alv13.1}. To this end let $O,f,u, \hat z, \rho, \mu_{f}$ be as in Theorem \ref{alv13.1}. Let $w, r, Q$ be as in Lemma \ref{prop} and let $\hat{\lambda}$ be a positive non-decreasing function on $(0,1]$ with $ \, \, \lim\limits_{t\to 0}\frac{\hat{\lambda}(t)}{t^{n-1}}=0. $
\begin{proposition}
\label{prop1}
There is a compact set $F=F(w, r) \subset  \partial O \cap B ( w, 20 r ) $ such that
 \[
\rh^{\hat{\lambda}}(F) = 0 \; \mbox{and}\; \mu_{f} ( B ( w, 100 r ))\leq c\mu_{f}( F ).
\]
\end{proposition}  
\begin{proof}
We first note from Lemma \ref{uismu} and the fact $u''\leq u'$ that for given $ j$, $1 \leq j \leq N$
\begin{align}
\label{3.31}
t_{j}^{1-n}\, \mu''_{f}(\overline{B}(z_j, t_j)) \leq c \, t_j^{1-p}\, \esssup\limits_{B(z_j , 2t_j )} (u')^{p-1} \leq c^2\,  t_j^{1-n}\,  \mu'_{f}(B ( z_j, 4 t_j) ) .
  \end{align}
where $N$ is the constant defined after \eqref{stoptime}.  For given $A>>1$, we see from (\ref{stoptime}) that $ \{ 1, 2, \dots, N \} $ can be divided into disjoint subsets: the good set $\mathfrak{G}$, the bad set $\mathfrak{B}$, and the ugly set $\mathfrak{U}$ as follows,
 \begin{align*}
  \left\{
  \begin{array}{l}
 \mathfrak{G}:=\{j: t_j >s\},\\
\mathfrak{B}:=\{j: t_j = s\; \mbox{and}\; \frac{f(\nabla u'')(x)}{| \nabla u''|(x)} \geq M^{- A }\; \mbox{for some}\; x \in \ar \Omega'' \cap \ar B ( z_j, s) \sem \La\} ,\\
 \mathfrak{U}:= \{j: t_j = s\; \mbox{and}\; j\not\in \mathfrak{B}\}.
  \end{array}
  \right.
 \end{align*}
Let $ t_j'= t_j $ when $ j \in \mathfrak{G} $ and $ t_j' = 4s $ when $ j \in \mathfrak{B}$. We define
  \begin{align}
  \label{defnofE}
  E:= \ar \Om' \cap \bigcup_{j\in \mathfrak{G} \cup \mathfrak{B}} B ( z_j, t_j' ).
  \end{align}
We first show for some $c=c(p,n,c_*)\geq 1$ and given $\epsilon > 0 $ that
\begin{align}
\label{hausdorffcontent}
\rh^{\hat{\lambda}}_{\delta'}(E)\leq \epsilon\; \mbox{and}\; c^{-1}\leq \mu'_{f}(E)
\end{align}
where $\delta'$ is as in (\ref{Mandmeasure}) and $\rh^{\hat{\lambda}}_{\delta'}(E)$ is the Hausdorff content of $E$ defined in \eqref{hausdorffcontentdefn}. Proposition \ref{prop1} will essentially follow from (\ref{hausdorffcontent}). To show (\ref{hausdorffcontent}), observe that if
    \begin{align}
    \label{3.32}
    x \in \bigcup_{j\in \mathfrak{G} \cup \mathfrak{B}} B ( z_j, t_j')\, \, \, \mbox{then}\, \, \, x\, \, \, \mbox{lies in at most}\, \, \, c = c ( n ) \, \, \, \mbox{of}\, \, \, \{ B ( z_j, t_j' ) \}.
  \end{align}
This observation can be proved using $ t_j \geq s, 1 \leq j \leq N, $  a volume type argument, and the fact that  $ \{ B ( z_j, t_j)\}_1^N $ is a Besicovitch covering of $ \ar \Omega' \cap \bar B (0, 15)$. 

We first consider $j\in\mathfrak{B}$. Using (\ref{nablau'epsilon}), (\ref{tk1-p}),  the  definition of $\mathfrak{B}$, and (\ref{3.32}) we find  for some $ c = c (p, n, c_*) \geq 1 $ that
  \begin{align}
  \label{MAt'j}
  M^{-A} \leq \frac{f(\nabla u'')(x)}{|\nabla u''( x ) |} \leq c (t'_{j})^{1-n}  \mu'_{f}(B(z_j, t'_{j})) \; \mbox{whenever}\; j\in\mathfrak{B}.
  \end{align}
Rearranging this inequality, summing over $j\in\mathfrak{B}, $ and using (\ref{3.32}), we see that
    \[
    \sum_{j\in \mathfrak{B}} (t_j')^{n-1} \, \leq \, \ti c \, M^A \mu'_{f} \left( \bigcup_{j \in \mathfrak{B}}  B ( z_j, t_j' ) \right)
    \leq (\ti c)^2 \, M^A
    \]
    provided $ \ti c = \ti c (p, n, c_*) $ is large enough. Now since $ t_j' = 4s $ for all
  $ j \in \mathfrak{B} $ we may for given $ A, M, \ep $ choose $ s > 0 $ so small that
 \begin{align*}
   \frac{\hat{\lambda}(4s)}{ (4s)^{n-1}}  \leq \frac{\ep}{ 2 (\ti c)^2 M^A }
\end{align*}
where we have used the definition of $\hat{\lambda}$. Using this choice of $s$ in \eqref{MAt'j} we get
  \begin{align}
  \label{3.34}
  \sum_{j\in \mathfrak{B}} \hat{\lambda} (t_j') \, \leq \, \frac{\ep}{2}.
  \end{align}

On the other hand, we may suppose $\delta'$ in (\ref{Mandmeasure})  is so small that
   $ \hat{\lambda}( t_j' ) \leq (t'_j)^{n-1} $ for $ 1 \leq j \leq N. $  Then from (\ref{u'ismu}), (\ref{stoptime}), and (\ref{3.32}), we see that
\begin{align}
 \label{3.35}
 \sum\limits_{j \in \mathfrak{G}}\hat{\lambda}(t_j')\leq \sum\limits_{ j \in \mathfrak{G}} (t_j')^{n-1}\leq \frac{1}{M} \sum_{j \in \mathfrak{G}} \mu'_{f}( B ( z_j, t_j ) ) \leq \frac{\ep}{2}
 \end{align}
 provided $ M = M ( \ep ) $ is chosen large enough.
  Fix $ M $ satisfying all of the above requirements. In view of (\ref{3.34}), (\ref{3.35}), and  the  definition of Hausdorff content we have $\rh^{\hat{\lambda}}_{\delta'}(E)\leq \epsilon\; \mbox{for}\; E\; \mbox{as in (\ref{defnofE})}$. This finishes the proof of the left hand inequality in (\ref{hausdorffcontent}). To prove the right hand inequality in \eqref{hausdorffcontent}, we use \eqref{prooff}, Lemma \ref{logfgradu}, and the definition of $ \mathfrak{U} $ to obtain
 \begin{align}
 \label{3.36}
 \begin{split}
   \mu''_{f} \left(\ar \Om' \cap \bigcup_{j\in \mathfrak{U}} \bar{B} ( z_j, t_j)\right) &\leq \mu''_{f}\left( \left\{ x \in \ar \Om'' : \frac{f(\nabla u'')(x)}{| \nabla u'' (x) |} \leq M^{-A} \right\}\right) \\
   & \leq \frac{c}{(p-1) A\log M} \int\limits_{\ar \Om''} |\log f(\nabla u'')| \, \frac{f(\nabla u'')}{| \nabla u'' |} \rd\rh^{n-1} \\
   &\leq \frac{c'}{A}.
   \end{split}
   \end{align}
  Choosing $ A=A(p,n,c_*) $ large enough we have from Lemma \ref{logfgradu} and (\ref{3.36}),
 \begin{align}
 \label{3.37}
 \mu''_{f} \left(\bigcup_{j\in \mathfrak{G} \cup \mathfrak{B}} B (0, 10) \cap \overline{B}( z_j, t_j) \right) \geq \mu'' ( B (0, 10 ) ) - \mu'' \left(  \bigcup_{j\in \mathfrak{U}} \bar{B} ( z_j, t_j)
 \right) \geq\frac{1}{c}
   \end{align}
   for some $ c(p, n, c_*)\geq 1$. Finally from (\ref{3.31})-(\ref{3.32}) and (\ref{3.37}),we get for some $ c = c(p, n, c_*) \geq 1 $ that
  \begin{align*}
  \frac{1}{c^{3}} \leq \frac{1}{c^{2}} \sum_{j \in \mathfrak{G} \cup \mathfrak{B}} \mu''_{f} (\bar{B} ( z_j, t_j) ) \leq \frac{1}{c} \sum_{j\in \mathfrak{G} \cup \mathfrak{B}} \, \mu'_{f} ( \bar{B} ( z_j, t'_j))\leq  \mu'_{f}( E ) .
   \end{align*}
For $j\in \mathfrak{G}$ we have used the definition of $t_j$ so that
\[
\mu'_{f}( B(z_j , 4t_j)) < M 4^{n-1} t_j^{n-1} =4^{n-1} \mu'_{f}( B(z_j,t_j)) =4^{n-1} \mu'_{f}(B(z_j,t'_j))
\]
Thus (\ref{hausdorffcontent}) is valid. To finish the proof of Proposition \ref{prop1}, we note that we can choose $E_{m}$ relative to $\delta'=\epsilon=2^{-m}$ for $m =m_{0}, m_{0}+1, \ldots$  with
\begin{align}
\label{Em}
\rh^{\hat{\lambda}}_{\delta'}(E_{m})\leq 2^{-m}\; \mbox{and}\; c^{-1}\leq \mu'_{f}(E_{m}).
\end{align}
From (\ref{Em}) and measure theoretic arguments we see that if we set
\[
E'=\bigcap\limits_{k=m_{0}}\left( \bigcup\limits_{m=k}E_{m}\right)
\]
then it follows from regularity of $\mu'_{f}$ that there exists a compact set $F\subset E'$ satisfying $\rh^{\hat{\lambda}}(F)=0\; \mbox{and}\; c^{-1}\leq\mu'_{f}(F)\; \mbox{where $c$ is as in (\ref{Em})}$. In view of these two estimates we conclude that the proof of Proposition \ref{prop1} is now complete
\end{proof}
We next give an easy consequence of Lemma \ref{prop} and Propositions \ref{prop1}. Let $Q$ be as in Lemma \ref{prop} and let $\hat{\lambda}$ be as in Proposition \ref{prop1}. We first prove that there exists a Borel set $ Q_1\subset Q$ with
  \begin{align}
  \label{3.39}
 \mu_{f} ( \partial O \cap B ( \hat z, \rho) \sem Q_{1})= 0\, \, \, \mbox{and}\, \, \, \rh^{\hat{\lambda}}(Q_{1}) = 0.
  \end{align}
To prove (\ref{3.39}) we first observe that if $\mu_{f}(\partial O\cap B(\hat{z},\rho))<\infty$ then it follows from Lemma \ref{prop}, Proposition \ref{prop1},  a Vitali type covering argument, and induction that there exists compact sets $\{F_{l}\}$ such that $F_{l}\subset Q,\; F_{k}\cap F_{l}=\emptyset\; \mbox{for}\; k\neq l$ and $\mu_{f}(F_{1})>0$ with
 \[
  \mu_{f} ( Q \sem \bigcup_{l = 1}^m F_l )\leq c'  \mu_{f} ( F_{m+1} ), \, \, m = 1, 2, \dots
 \]
 for some $c' = c' (p, n, c_*) \geq 1$.  Moreover  $ \rh^{\hat{\lambda}} ( F_l ) = 0 $ for all $l.$ Then it follows from measure theoretic arguments that $ Q_1 = \bigcup_{l=1}^\infty F_l$ has the desired properties in \eqref{3.39}. In case $\mu_{f}(\partial O\cap B(\hat{z},\rho))=\infty$, we can write $\partial O\cap B(\hat{z},\rho)$ as a union of countable Borel sets with finite $\mu_{f}$ measure and apply the same argument in each set.
Therefore we conclude that there exists a Borel set $Q_{1}$ in $Q$ satisfying (\ref{3.39}).

We now prove Theorem \ref{alv13.1}. To this end, we let
  \[
  P := \left\{ x \in \ar O \cap B ( \hat z, \rho ):\; \limsup_{t \to 0} \frac{\mu_{f} ( B ( x, t ) )}{t^{n-1}} > 0 \right\}.
   \]
   We first show that
   \begin{align}
  \label{3.40}
  \mu_{f}(\ar O \cap B ( \hat z, \rho ) \sem P ) = 0.
  \end{align}
From Lemma \ref{prop} we have $ \mu_{f}(\ar O \cap B ( \hat z, \rho ) \sem Q ) = 0 $. Therefore, it suffices to prove \eqref{3.40} with $ Q $ replacing
 $ \ar O \cap B ( \hat z, \rho ). $  To do this we argue by contradiction and thus assume $\mu_{f}( Q \sem P ) > 0$. Then, by Egoroff's theorem there exists a compact set $ K \subset Q \setminus P $ with
  \begin{align}
  \label{3.41}
  \mu_{f} (K)> 0\; \mbox{and}\; \lim_{t \to 0} \frac{\mu_{f} (B ( x, t ) )}{t^{n-1}} = 0 \; \mbox{uniformly for}\; x \in K.
  \end{align}
  Set $ \al_0 = 1$ and choose $ \al_k \in (0, 1), k = 1, 2, \dots, $ such that
  \[
 \al_{k+1} < \frac{\al_k}{2}\; \mbox{and}\; \sup_{0 < t \leq \al_k }  \frac{\mu_{f} ( B ( x, t ) )}{t^{n-1}} \leq 2^{-2k} \mbox{ for all } x \in K.
  \]
Define $ \hat{\lambda}_{0} ( t ) $ on $(0, 1]$ in the following way:  put $ \hat{\lambda}_{0} (0) = 0$,
\[
\hat{\lambda}_{0} ( \al_k ) = 2^{-k} ( \al_k)^{n-1}\; \mbox{for}\; k = 0, 1, \dots,
\]
and define $\hat{\lambda}_{0}(t)$ when $t\in[\al_{k+1}, \al_k ]$  in such a way that
\[
\frac{\hat{\lambda}_{0}(t)}{t^{n-1}}\; \mbox{is linear for}\; t\in [\al_{k+1}, \al_k ]\;  \mbox{whenever}\; k = 0, 1, \dots.
 \]
 Clearly,
  \[
  \frac{\hat{\lambda}_{0} ( t )}{t^{n-1}} \to 0\; \mbox{as}\; t \to 0.
  \]
  Moreover, we observe that
  \begin{align}  
  \label{3.41a} 
\begin{split}
&  \hat{\lambda}_{0}(2t)\leq 2^{n+1}\hat{\lambda}_{0} (t)\; \mbox{for}\; 0<t<1/2,\\
&    \frac{ \mu_{f}( B ( x, t ) ) }{ \hat{\lambda}_{0} (t) } \leq  2^{1 - k}\; \mbox{whenever}\; \al_{k+1} \leq t \leq \al_k\; \mbox{and}\; x \in K.
  \end{split}
  \end{align} 
Let $ Q_1 $ be as in \eqref{3.39} relative to  $ \hat{\lambda}_0. $  Then for a given positive integer $m$ it follows from (\ref{3.39}) that there is a covering $ \{ B ( x_j, r_j ) \}$ of $ K\cap Q_{1}$ with
  \[
r_{j} \leq\frac{\al_m}{2} \;\mbox{for all}\; j\; \mbox{and}\;  \sum\limits_{j} \hat{\lambda}_{0} ( r_j ) \leq 1.
  \]
We may assume that there is an $ x_j' \in K \cap B ( x_j, r_j ) $ for each $ j $ since otherwise we discard $ B ( x_j, r_j ). $   Then from (\ref{3.41a})  we find that
  \[
  \mu_{f} ( K \cap Q_{1}) \leq \sum\limits_{j} \mu_{f} ( B ( x_j', 2 r_j ) ) \leq 2^{1 - m}  \sum\limits_{j} \hat{\lambda}_{0} (2 r_j ) \leq 2^{n + 2 -m} .
  \]
Since $m$ is arbitrary, we have $\mu_{f}(K \cap Q_1) = 0. $ Using this equality and \eqref{3.39} we find that $ \mu_f ( K ) = \mu_f ( Q\sem Q_1) + \mu_{f}(K\cap Q_{1}) = 0$  and so  we have reached a contradiction in \eqref{3.41}. Hence,  $\mu_{f}( Q \sem P ) =0$ and (\ref{3.40}) holds.

We next show that the set $P$ has $ \sigma-$ finite $\rh^{n-1} $ measure. To this end, once again we may assume  $ \mu_{f} ( \ar O \cap B ( \hat z , \rho ) ) < \infty$. Let $m$ be an arbitrarily fixed positive integer and define
   \[
   P_m := \left\{ x \in P:\; \limsup_{t \to 0}   \frac{\mu_{f}( B ( x, t ) )}{t^{n-1} } > \frac{1}{m} \right\}.
  \]
 Given $ \hat{\delta} > 0 $ we choose a Besicovitch covering $ \{ B ( y_i, r_i ) \}$  of $P_m$ with
\[
   y_i \in P_m,\; r_i \leq \hat{\delta},\; B ( y_i, r_i ) \subset B ( \hat z, \rho ), \; \mbox{and} \; r_i^{n-1} <m\, \mu_{f} ( B ( y_i, r_i ) ).
  \]
It follows that
   \begin{align}
   \label{3.43}
   \sum\limits_{i}r_i^{ n - 1} < m \sum\limits_{i} \mu_{f} ( B ( x_i, r_i ) )\leq c \, m
    \mu_{f} (\ar O \cap B ( \hat z, \rho ) ) < \infty.
    \end{align}
          Letting $ \hat{\delta} \to 0 $ and using the definition of $\rh^{n - 1} $ measure we conclude from (\ref{3.43}) that $\rh^{n - 1} ( P_m ) < \infty$. As $m$ is arbitrary we conclude that $P$ has $\sigma-$finite $\rh^{n-1} $ measure.
  In view of this observation, (\ref{3.40}) and Lemma \ref{prop}, the  proof of Theorem \ref{alv13.1} is now  complete. $ \Box $ 
\section{Proof of Theorem \ref{alv13.2}}
\label{section5}
This section is dedicated to the proof of Theorem \ref{alv13.2}. Before giving a proof we recall our setting from section \ref{intro}; let 
\[
\ti \Ga = \left\{\ti Q_{k, j};\;  k = 1, \dots, \;\mbox{and}\; j = 1, \dots, 2^{kn} \right\}
\]
denote the set of cubes defined in section \ref{intro} and let
$\m{C}$ be the corresponding Cantor set.  Also as in section \ref{intro} let  $ \m{S} $ be the cube in $\mathbb{R}^{n}$ with side length 1 centered at the origin.  and let $u^{\infty}$ be the positive weak solution to
\[
\Delta_{f} u^{\infty}=\nabla \cdot \m{D}f(\nabla u^{\infty})=0 \; \mbox{in}\; \m{S}\setminus \m{C}
\]
with continuous boundary values $1$ on $\partial \m{S}$ and $0$ on $\m{C}$.  Let $\mu^{\infty}_{f}$ be the measure associated with $u^{\infty}$ as in (\ref{ast}). For ease of notation, we write $ \mu, u $ for $ \mu_f^\infty, u^\infty. $  Next let $\alpha, \beta$ be the constants as in section \ref{intro} and $s(\tilde{Q}_{k,j})=a_{0}a_{1}a_{2}\ldots a_{k} < 2^{-(k+1)}$ denote the side length of $\tilde{Q}_{k,j}$ where $\alpha\leq a_{i}\leq \beta<1/2$ for every $i=1,2,\ldots$. Let $c_*$ be the constant as in \eqref{prooff}.  

Let $ \ti Q \in \ti \Ga $  for some $k$ with $ k \geq 10^5$ and $j=1,\ldots, 2^{kn}$. We first show that
\begin{align} 
\label{new1} 
\mu ( 100 \ti Q ) \leq c \max_{ \ar 2 \ti Q } u^{p-1} \leq c^2 \mu ( \ti Q )\; \mbox{for some}\; c=c(p, n, c_*, \al, \be )
\end{align} where once again $ c_* $ is as in \eqref{prooff}.  To prove \eqref{new1} note from the geometry of $\m{C}$ that there exists a smallest  $ \ti Q'  \in \ti \Ga $ with
 \[   
 100 \ti Q \subset ( 1 + \he ) \ti Q' 
 \]  
 where $ \he = \frac{1}{100n} \min ( \al, 1/2 - \be )$.  Covering $ \ti Q' \cap  \mathcal{C} $ by balls  of radius $ \approx s ( \ti Q) $ and applying Lemma \ref{uismu} in each ball we deduce that
  \begin{align}
  \label{eqn11}
  \mu (100 \ti Q)  \leq \mu ( \ti Q' )  \leq c \max_{\ar (1+\he ) \ti Q'} u^{p-1}
  \end{align}
   where $ c = c (p, n, c_*, \al, \be )$. Using Harnack's inequality, basic geometry and once again Lemma \ref{uismu} we also see that
  \begin{align}
  \label{eqn12}
  (s(\ti Q))^{n-p} \max_{\ar (1+\he ) \ti Q'} u^{p-1} \leq \ti c (s(\ti Q))^{n-p}  \max_{\ar (1+\he ) \ti Q} u^{p-1} \leq \ti c^2 \mu (\ti Q ) 
\end{align}
   where $ \ti c $ has the same dependence as $c$. Combining \eqref{eqn11} and \eqref{eqn12} we obtain
   (\ref{new1}). From H\"{o}lder continuity of $1 - u$ near $\ar\m{S}$,  Harnack's inequality, and Lemma \ref{uismu} we also find that
 \begin{align} 
 \label{new2}  
 \mu ( \m{C} ) \approx 1
 \end{align} 
 where proportionality constants depend only on $ p, n, c_*, \al, \be$. Analogous to Proposition \ref{prop1} we prove
   \begin{proposition}
\label{prop2}
Let $ \ti Q \in \ti \Ga $ be  a given cube. Then there exists $ \de' > 0 $ with the same dependence as $ \de $ in Theorem \ref{alv13.2},  $ c = c ( p, n, c_*, \al, \be ) \geq 1$, and a compact set $F \subset \mathcal{C} \cap \ti Q $  with 
\[
\rh^{ n-1 - \de' }( F ) = 0 \; \mbox{and}\; \mu ( \ti Q )\leq c \mu ( F ).
\]
\end{proposition}
\begin{proof}
 We shall only show that the conclusion of Proposition \ref{prop2} is valid with $ \ti Q $ replaced by $ \ti Q_0 = $ the closed cube with side length 1/2 and center at 0 (denoted $\m{C}_{0}$ in section \ref{intro}). The general version of Proposition \ref{prop2} is proved in a similar way, as one sees from using (\ref{new1}) and arguing as in the construction of $ u' $ in (\ref{u'ismu}),
With this understanding, we simplify the proof of Proposition \ref{prop2} further by noting that if $ \la (r) = r^{n - 1 - \de'},\; 0 < r \leq 1$, then from measure theoretic type arguments  it suffices to show for given $ \ep, \tau > 0, $ that there exists $ \de', c $ as above and a compact set $ F \subset \m{C} $ with
\begin{align} 
\label{new3} 
\rh^{\la}_{\tau} ( F ) \leq \ep\; \mbox{and}\; \mu ( F ) \geq 1/c. 
\end{align} 
 To prove \eqref{new3} and in view of the proof of Theorem \ref{alv13.1} we shall need  some more notation: Let
$ \{B(x_l, \frac{\he}{10} )\}_1^{N_1}, x_l \in \ar \ti Q_{0},  $ be a Besicovitch covering of $ \ar \ti Q_0 $ and set
\[ 
Q_{0}:=\tilde{Q}_{0}\cup \left( \bigcup_{l=1}^{N_1} B(x_l, \frac{\he}{10} ) \right) 
\]
 where $ \he = \frac{1}{100n} \min ( \al, 1/2 - \be ) $ as earlier (see figure \ref{figureQ0}). 
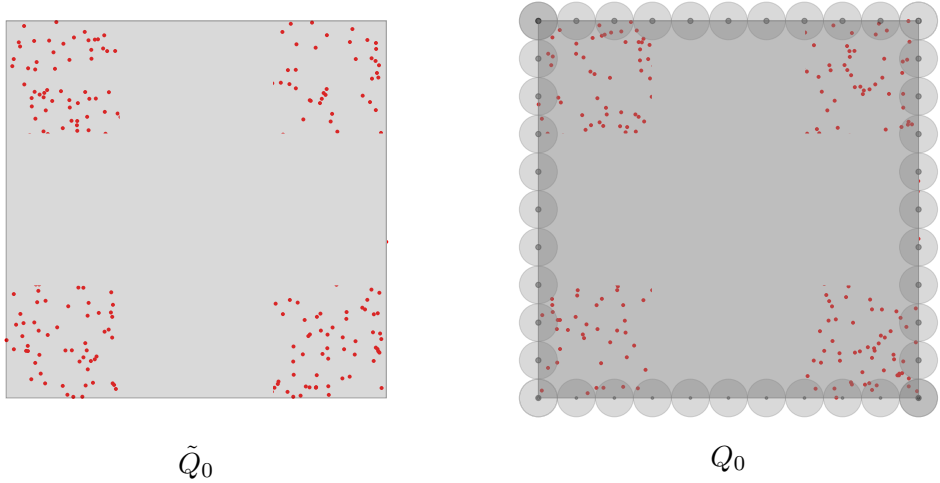
\begin{figure}[!ht]
\centering
\begin{tikzpicture}
\begin{scope}[shift={(-5,0)}]
\foreach \i in {1,...,500}
  \fill[red, opacity=2] (rnd*5cm, rnd*5cm) circle (.75pt);
  \draw[white, fill=white] (1.5,0) rectangle (3.5,5);
  \draw[white, fill=white] (0,1.5) rectangle (5,3.5);  
  \draw[fill=gray, opacity=0.3] (0,0) rectangle (5,5);
  \node[below] at (2.5,-.5) {\mbox{$\tilde{Q}_{0}$}};
  \end{scope}
  \begin{scope}[shift={(2,0)}]
\foreach \i in {1,...,500}
  \fill[red] (rnd*5cm, rnd*5cm) circle (.75pt);
    \draw[white, fill=white] (1.5,0) rectangle (3.5,5);
  \draw[white, fill=white] (0,1.5) rectangle (5,3.5);  
  \filldraw[gray, opacity=.3] (0,0) rectangle (5,5);
\foreach \x in {0, 0.5, ..., 5}
\filldraw[gray, opacity=.3] (\x, 0) circle (.25cm) (\x, 5) circle (.25cm)  (0, \x) circle (.25cm)  (5, \x) circle (.25cm);
\foreach \x in {0, 0.5, ..., 5}
\filldraw[opacity=.3] (\x, 0) circle (.5pt) (\x, 5) circle (1pt)  (0, \x) circle (1pt)  (5, \x) circle (1pt);
  \filldraw[gray, opacity=0.3] (0,0) rectangle (5,5); 
 \node[below] at (2.5,-.5) {\mbox{$Q_{0}$}};
  \end{scope}
\end{tikzpicture} 
 \caption{The cubes $\tilde{Q}_{0}$ and $Q_{0}$.}
\label{figureQ0}
\end{figure}

 If  $  \ti Q $ is a cube with center $z$ let  $ \ga Q =  \{ x = z + 2 \ga s( \ti Q) y : y \in Q_0 \}. $  We write $ Q $ for
$1Q. $  From our constructions we have for $ k = 1, 2, \dots, \, $
\begin{align}
\label{cubesaredisjoint}
\begin{split} 
&  \ti Q_{k, j } \subset  Q_{k, j} \subset (1+\he/2) \ti Q_{k,j}  \\
& (1+\he) \ti Q_{k,j} \cap (1 + \he) \ti Q_{k,j'}=\emptyset\; \mbox{for}\; j\neq j' \mbox{ and }\\
&\mbox{either}\; \ti Q_{k,j} \subset \ti Q_{k',j'} \; \mbox{or} \; \ti Q_{k,j} \cap \ti Q_{k',j'} = \es,  \mbox{ when }\; k>k' .  \end{split}
\end{align}
Let $ \ti \La $ be a finite disjoint covering of $ \m{C} $ by cubes in
$ \ti \Ga $ and let  $ \La $ be the collection of all $ Q_{k,j} $  with  $
\ti Q_{k,j}  \in \ti \La. $

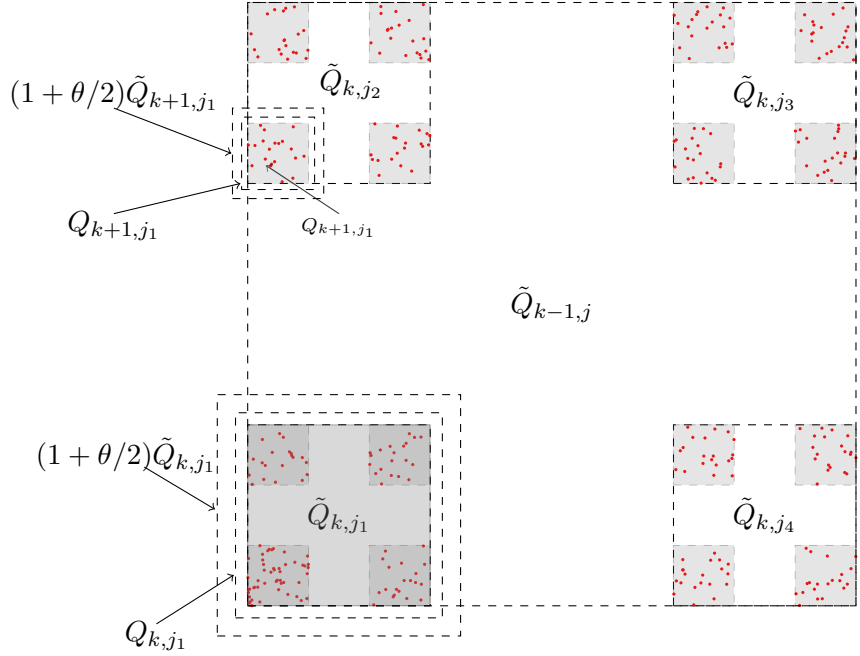
\begin{figure}[!ht]
\centering
\begin{tikzpicture}[scale=.8]
\foreach \w in {0,7}
\foreach \v in {0,7}
{
\draw[dashed] (\w,\v) rectangle (3+\w,3+\v);
}
\begin{scope}
\foreach \i in {1,...,30}
  \fill[red] (rnd*1, rnd*1) circle (.75pt);
  \end{scope}

\foreach \j in {0,2,7,9}
\foreach \k in {0,2,7,9}
{
\begin{scope}[shift={(\j,\k)}]
\foreach \i in {1,...,20}
  \fill[red] (rnd*1, rnd*1) circle (.75pt);
  \end{scope}
 } 
\foreach \w in {0,2,7,9}
\foreach \v in {0,2,7,9}
{
\draw[dashed, fill=gray, opacity=.2] (\w,\v) rectangle (1+\w,1+\v);
}
\draw[dashed] (-.2,-.2) rectangle (3.2, 3.2);
\draw[dashed] (-.5,-.5) rectangle (3.5,3.5);
\node at (5,5) {$\tilde{Q}_{k-1,j}$};

\node at (1.5,1.5) {$\tilde{Q}_{k,j_1}$};
\node at (1.7,8.7) {$\tilde{Q}_{k,j_2}$};
\node at (8.5,8.5) {$\tilde{Q}_{k,j_3}$};
\node at (8.5,1.5) {$\tilde{Q}_{k,j_4}$};
\node at (-2, 2.5) {$(1+\theta/2)\tilde{Q}_{k,j_{1}}$};
\node at (-1.5, -.5) {$Q_{k,j_{1}}$};
\draw[->] (-1.7,2.3)--(-.54,1.6);
\draw[->] (-1.5,-.3)--(-.23,.51);
\draw[fill=gray, opacity=.3] (0,0) rectangle (3,3);
\draw[dashed] (-.1,6.9) rectangle (1.1, 8.1);
\draw[dashed] (-.25,6.75) rectangle (1.25,8.25);
\node at (-2.2, 8.5) {$(1+\theta/2)\tilde{Q}_{k+1,j_{1}}$};
\node at (-2.2, 6.3) {$Q_{k+1,j_{1}}$};
\draw[->] (-2.2,8.25)--(-.27,7.5);
\draw[->] (-2.2,6.5)--(-.12,7);
\draw[dashed] (0,0) rectangle (10,10);
\draw[->, opacity=.5] (1.5,6.5)--(.3,7.3);
\node[below] at (1.5,6.6) {\tiny\tiny\mbox{$Q_{k+1, j_{1}}$}};
\end{tikzpicture} 
 \caption{The cubes $\tilde{Q}_{k-1,j_{1}}$, $\tilde{Q}_{k,j_{1}}$, $\tilde{Q}_{k+1,j_{1}}$, $Q_{k,j_{1}}$, and $Q_{k+1,j_{1}}$.}
\label{figureQkj1}
\end{figure}

\begin{remark}
Note that cubes $\tilde{Q}_{k,j}\in\ti \Lambda$ are closed cubes whereas the cubes $Q_{k,j}\in\Lambda$ are open. Moreover, 
figure \ref{figureQkj1} tells us (as an example) that $\ti Q_{k-1,j}\notin\ti\Lambda$ and $\ti Q_{k,j_{2}}\notin\ti\Lambda$. On the other hand, $\ti Q_{k,j_{1}}\in\ti\Lambda$ and $\ti Q_{k+1,j_{1}}\in\ti\Lambda$ and therefore by definition of $\Lambda$, $Q_{k,j_{1}}\in\Lambda$ and $Q_{k+1,j_{1}}\in\Lambda$. 
\end{remark}
 Let  $ \bar u $ be the
 positive weak solution to
\[
\Delta_{f} \bar u=\nabla \cdot \m{D}f(\nabla \bar u)=0 \; \mbox{in}\; \Om =  B (0, n ) \setminus \bigcup_{Q\in \La}  \bar Q
\]
with boundary values $1$ on $\partial B (0, n ) $ and $0$ on $ \ar Q$ for every $Q \in\La$. Extend  $ \bar u $ to $ B(0, n) $ by putting  $ \bar u = 0 $ on every $Q\in\La$. Let $\bar \mu$ be the measure associated with $ \bar u
$ as in (\ref{ast}). Let $ \bar v = \log f ( \nabla \bar u ) $ and define $ \ti L $ as in
Lemma \ref{v'issolnlemma} relative to $ \bar u. $ Recall from this lemma that $ \max ( \bar v, \eta ) $ is a weak sub solution to $ \ti L$ whenever $ \eta \in ( - \infty, \infty ). $
 Then $ \ti L \bar v = \nu $ weakly, where $ \nu $ is a locally finite positive Borel measure on  $ \Om \cap \{x:\; | \nabla \bar u | > 0 \}. $ In case $ p = 2, n = 2, $ we shall see that $ \nu $    is   a   locally finite atomic measure on $ \Om. $  

Next we state a key lemma.
\begin{lemma} 
\label{finess} 
Let $ \Om, \bar u, \bar \mu, \bar v, \nu, $ be as above and suppose $ \ti Q \in \ti \Ga \sem \ti \La. $  There exists  $ c_2, c_3, c_4   \geq 10^{5} $,  such that if
$ \ti Q'  \subset \ti Q, \, \ti Q'  \in \ti \La, $ and $ c_2 s ( \ti Q' ) \leq s ( \ti Q),$ then
\[ 
\int_{O} \bar u \rd \nu  \geq c_3^{-1} \bar \mu ( \bar Q ) \; \; \mbox{where}\;\; O = \left\{x \in  (1 +\he) \ti Q:\;  d ( x, \ar \Om ) \geq \frac{ s(\ti Q )}{c_4} \right\}
\] 
and $Q\in\Lambda$ is the cube associated with $\ti Q\in\ti\Lambda$. Here $c_2, c_4 $ depend only on $ p, n, c_*, \al, \be, $ and can be chosen independent of $p \in [n, n + 1]$ provided $ c_* $ in \eqref{prooff} is constant on this interval.  Also $ c_3^{-1} \geq (p - n ) c_5^{-1}, $ where $ c_5 $ has the same dependence as $ c_2. $ 

Moreover, if $ f = g^p, $ where $ g $ is as in Theorem \ref{alv13.2}, then $ c_3^{- 1} $ can be chosen to depend only on $ n, g, c_*, \al, \be, $ when $ p \in [n, n + 1]. $
\end{lemma}

\begin{proof}  
Let  $ \xi $ be the minimum of $ \bar u $ on  $ \ar(1+\he) \ti Q $ and let  
\[
G = \left\{x:\; \bar u ( x ) < \xi/ 2 \right\} \cap ( 1 + \he ) \ti Q. 
\]
We note that   
\[
d ( \ar G, \ar \Om )  \geq  c^{-1} s ( \ti Q )\; \mbox{and}\; \xi  \geq c^{-1} \max\limits_{ (1 + \he) \ti Q } \bar u 
\]
   thanks to  Harnack's inequality and  H\"{o}lder continuity of $ \bar u $ near $ \ar \Om $ (Lemmas \ref{fislegthanu} and \ref{uisholder}). From this note and our hyphothesis, we deduce that if   $ c_2 $ is large enough, then a component of $ G, $ say $ G', $ contains two disjoint cubes,  $ (1+ \he ) \ti Q_1, (1+ \he )  \ti Q_2, $  with $ \ti Q_1, \ti Q_2  \in \ti \Ga $  and
 \[  
 (c')^{-1}  s ( \ti Q_i ) \leq s ( \ti Q ) \leq c'  s ( \ti Q_i )\; \mbox{for}\; i = 1, 2, 
 \]
 where $ c' $ has the same dependence as $ c_2 $ in Lemma \ref{finess}.   Let  $ \xi_1 $ be the minimum of
$  \bar u $ on $  \ar ( 1 + \he ) \ti Q_1 \cup \ar (1+\he ) \ti Q_2.  $  Then
from our construction, the maximum principle for solutions to \eqref{flaplace},  and once again
Harnack's inequality - H\"{o}lder continuity of $ \bar u $ near $ \ar \Om, $ we see that  $ G' $  contains at least two components of $ G_1 =  \{ x:\;  \bar u ( x ) < \xi_1/2 \}. $  Moreover,
\begin{align} 
\label{new5}   
d ( \ar G_1, \ar \Om )   \geq  c^{-1} s ( \ti Q )  \mbox{ and } \xi_1   \geq c^{-1} \max_{ (1+\he) \ti Q } \bar u .  
\end{align} 
Let $ t_0$, $\xi_1/2   \leq t_0  <  \xi/2, $ be the largest $ t $ for which   there are at least  two components of  $ \{x:\; \bar u ( x ) < t \} $ contained in $ G'. $
Then there exists  $ \hat x \in  G' \cap \{x:\; \bar u ( x ) = t_0 \} $ such that $ \hat x $ lies on the boundary of  two different
components of $ \{x:\; \bar u ( x ) < t_0 \} $ in  $ G'. $  Also,
\begin{align} 
\label{new6} 
\bar u (\hat x) = t_0, \; \nabla \bar u ( \hat x ) = 0,\; \, d ( \hat x, \ar \Om ) \geq c_6^{-1} s ( \ti Q),\; \mbox{and}\; t_0 \geq c_6^{-1} \max_{(1+\he) \ti Q} \bar u,  
\end{align} 
 where $ c_6 \geq 1 $ has the same dependence as $ c_4 $ in  Lemma \ref{finess}. Indeed observe from Lemma \ref{localholderfornablau}  that $ \bar u $ has H\"{o}lder continuous derivatives in an open neighborhood of $ \hat x. $ So if $ \nabla \bar u (\hat x ) \not = 0 $ we  easily obtain a contradiction to the definition of $ t_0 $, using the implicit function theorem
 and the definition of a component.   From this contradiction we conclude that $ \nabla \bar u ( \hat x ) = 0. $ Existence of $ c_6 $ depending on $ p, n, c_{*}, \al, \be $ follows from \eqref{new5} which in turn was  proved using  Lemmas \ref{uisholder} and \ref{fislegthanu}. Also it is  easily checked from references providing proofs of these lemmas (see section \ref{preplemmas}) that constants may be chosen to  depend only on $ n, \al, \be $ when $ p \in [n, n + 1] $  provided $ c_* $  in \eqref{prooff} is chosen independent of $p$ in this interval.

For $ \hat x, t_0 $ as in \eqref{new6} we now  choose 
\[
z \in \ar \Om \sem \ar B (0, n)\; \mbox{with}\; d ( \hat x, \ar \Om \sem \ar B(0, n ) ) = | z - \hat x |. 
\]
Let $ z_1 $ be the first point on this line segment starting from $ \hat x $  with  $ \bar u ( z_1 ) = (1/2)t_0. $ Let
 $ [\hat x, z_1] $ denote the line segment from $ \hat x$ to $ z_1.$
Then   
\[
  (1/2)t_0\leq  \int_{[\hat x, z_1]} |\nabla \bar u | \rd\m{H}^1 
  \] 
  so there exists $ z_2 $ on $ [\hat x, z_1] $ and $ c_7 = c_7 (p,n,c_*,\al, \be) $  with
\begin{align} 
\label{new7} 
(1/2) t_0 \leq  | \nabla   \bar u ( z_2 ) | | \hat x - z_1 |\; \mbox{while}\;  d ( z_2, \ar \Om ) \geq c_7^{-1} s(\ti Q) 
\end{align} 
where the last inequality follows from our choice of $ z_1, $ basic geometry, and Lemma \ref{uisholder}. From \eqref{new6}, \eqref{new7},  we find $\rho$ such that
\[
 \rho=  \rho ( p, n, c_*, \al, \be ) \geq c^{ - 1}  s ( \ti Q )\;  \mbox{with}\; B ( \hat x, \rho ), B ( z_2, \rho ) \subset \Om. 
 \]
 Let  $ \Om' $ denote the convex hull  of $B(\hat x, \rho/2 )$ and $B ( z_2, \rho/2 )$.  Then from  Harnack's inequality, Lemma \ref{localholderfornablau}, a Poincare type inequality, and \eqref{new6}, \eqref{new7}, we have
\begin{align}  
\label{new8} 
\begin{split}
c^{-1} s ( \ti Q )^{n-p} ( \max_{(1+\he) \ti Q}  \bar u )^{p-1} & \leq \,   s ( \ti Q )^{n-p}    \bar u ( z_2 )^{p-1}  \\
&\leq c \int_{\Om' \cap \left\{ |\nabla \bar u | > 0 \right\} } \, \bar u |\nabla \bar u |^{p-2} \, |\nabla \bar v|^{2} \, \rd x \\
  & \leq  c^2  \int_{\Om' \cap \left\{ |\nabla \bar u | > 0 \right\}} \, \bar u  |\nabla \bar u|^{p-4} \, \sum_{i, j = 1}^n \bar u_{x_i x_j}^2 \, \rd x. 
\end{split}  
  \end{align} 
Using  Lemma \ref{uismu}  and  \eqref{new8}, it follows that    
\begin{align} 
\label{new9} 
\bar \mu ( \bar Q ) \leq  c   \int_{\Om' \cap \left\{ |\nabla \bar u | > 0 \right\}} \, \bar u  |\nabla \bar u|^{p-4} \, \sum_{i, j = 1}^n \bar u_{x_i x_j}^2 \, \rd x.
\end{align} 
Next we revisit the proof of Lemma \ref{v'issolnlemma} in order to estimate $ \nu$. 
\subsection{The case $p\geq n>2$}
In this case from \eqref{I'andI''}-\eqref{add1} we see for $ n > 2 $ and $ p \geq n $ that  if  $  \nabla  \bar u (x)  \not = 0, $ then 
\begin{align} 
\label{new10}  
\ti L \bar{v}=h\; \mbox{weakly}
\end{align}
where
\[
h =  f^{-1} ( I' + I'') = f^{-1} \, [ \, \mbox{tr}(BA)^2 - \frac{1}{f}\frac{1}{(p-1)^2} \nabla \bar{u} BABAB (\nabla \bar{u})^t \, ]
\]
and $ A = ( \bar u_{x_i x_j} ), B = ( f_{\eta_i \eta_j} ) $ are $ n \times n $ matrices. If $p>n$ we see from \eqref{trofeiszero} and  \eqref{final-1} that
\begin{align} 
\label{new11}   
\mbox{tr}(BA)^2 -\frac{1}{f}\frac{1}{(p-1)^2} \nabla \bar{u} BABAB (\nabla \tilde{u})^t   \geq  \frac{ p - n }{ n (n -1)(p-1)} \mbox{tr}(AB)^2 . 
\end{align} 
Moreover,  
\begin{align} 
\label{new12} 
\hat c\; \mbox{tr}(AB)^2  \geq  |\nabla \bar u |^{2p - 4}  \sum_{i, j = 1}^n  \bar u_{x_i x_j }^2 
\end{align} 
for some $ \hat c \geq 1 $ depending only on $ p, n, \al, \be, $ and $ c_* $ in \eqref{prooff}, as follows from positive definiteness and $ p - 2 $ homogeneity of $ B $ as well as symmetry of $ A $. Combining \eqref{new10}-\eqref{new12} we conclude for almost every $ x $ with $ \nabla \bar u (x) \not = 0 $ that
\begin{align} 
\label{new13} 
h \geq  c^{-1} (p - n) |\nabla \bar u |^{p - 4} \,  \sum_{i, j = 1}^n  \bar u_{x_i, x_j }^2 
\end{align} 
where $ c \geq 1 $ depends only on $ p, n, \al, \be,$ and $ c_* $ in \eqref{prooff}. Combining \eqref{new13}, \eqref{new10}, and \eqref{new9} we get
\begin{align} 
\label{new14}  
( p - n )\bar \mu ( \bar Q ) \leq \ti c \int_{\Om' \cap \left\{ |\nabla \bar u | > 0 \right\}  } \bar u\, \rd \nu \leq \tilde{c}
\int_{O} \bar u \rd \nu 
\end{align} 
where $ \ti c = \ti c ( p, n, c_*, \al, \be)$ and this constant can be chosen independent of $ p $ on $ [n, n+1]. $   From the definition of $ \Om' $ and \eqref{new14} we see that the first part of  Lemma \ref{finess} is true when $p>n$.  

To handle values of $ p$ near $ n$, $n \geq 3, $  we need to examine the case when  $ h = 0 $  (so $ p = n $) in  \eqref{new10}. Indeed, from  \eqref{trofeiszero} -  \eqref{final-1}   we see  for $ p = n $ that 
\begin{align*} 
f h  =  \left[  \mbox{ tr }(E^2) -  \frac{n}{n-1} \frac{ y E^2 y^t}{y y^t} \right] =  g (y)  \geq 0, 
\end{align*} 
 where  $ E = B_d' A_1 B_d'$,  $y  = \nabla \bar u  \m{S}  B_d'  $, $ A_1  = \m{S}^t A \m{S}$,  $B_d = \m{S}^t B  \m{S}, $ and $  B_d  = B'_d  B'_d $. Also $\m{S} $   is an  orthogonal matrix and $B_d $ a diagonal matrix as in section \ref{subsol}. If  $ g  (y)  =  0, y \not = 0,  $   and  
  $ E \not = 0 $, then since $ E $  is symmetric, it follows from basic matrix theory that  $ y $ is an eigenvector of $ E, $ so $ y E  = \mathcal{V} y $ for some $ \mathcal{V}\neq 0. $ Thus,
\[    
( \nabla \bar u \m{S}  B_d' ) B_d' A_1 B_d' =  \mathcal{V} ( \nabla \bar u \m{S}  B_d' )   
\] 
so since $ \m{S}, B_d' $ are invertible it follows that at $x$ we have $\nabla\bar u B  A = \mathcal{V} \nabla \bar u$. If we rewrite this in terms of $f$ and $\bar{u}$ we get
 \begin{align} 
 \label{new16}  
 (n-1) \nabla f ( \nabla \bar u ( x ) ) = \mathcal{V} \nabla \bar u ( x )  
 \end{align} 
where we have used the $n - 1 $ homogeneity of $ \m{D} f$. On the other hand at almost every $ x $ where  $ E = 0 $ we have
\begin{align}  
\label{new17}  
 A = ( \bar u_{x_i x_j } ) = 0\;  \mbox{since}\; B_d'\; \mbox{and}\; \m{S}\;\mbox{are invertible}. 
 \end{align} 
Assume that either \eqref{new16} or \eqref{new17} hold almost everywhere  in $ B(w, r) \subset \Om$ for some $w\in\Omega$ and $r>0$.  If $B( w, r ) \cap \{ x : \bar u ( x ) = t \} \neq \es,$  we assert that  
\begin{align}  
\label{new18}  
f( \nabla \bar u )\; \mbox{is constant on each component of}\; B ( w, r ) \cap \left\{ x : \bar u ( x ) = t \right\}. 
\end{align} 
To prove this assertion let $ x'  \in B ( w, r )$ and suppose that $ \nabla \bar u (x') \not = 0. $  Then 
\[
\pm \bar u_{x_i} (x') \geq n^{-1}  | \nabla \bar u ( x' ) | \; \mbox{for some}\; 1 \leq i \leq n. 
\]
Assume for example that $ i = n $ so that $ \bar u_{x_n} (x')  \geq |\nabla \bar u | (x')/n. $  Consider the mapping, $\Psi ( x_1, \dots, x_n ) = (x_1, \dots, \bar u ( x_1, ..., x_n ))$.  From the inverse function theorem and  Lemma \ref{localholderfornablau}  we see that in a neighborhood of $ \Psi (x'),  \Psi $ has a $ C^{1, \al'' } $ inverse
 $ \Ph $  and  $ f ( \nabla \bar u (x) ) $ is in $ W^{1,2} ( B ( x', \rho ) ) $ for some small $ \rho > 0. $ We claim that
 \begin{align}
 \label{new181}
 H =  f ( \nabla \bar u ) \circ \Ph\in W^{1,2} ( B ( \Psi (x'), \rho' ) )\; \mbox{for small}\; \rho' > 0. 
 \end{align}
 Here $H$ is considered as a function of $ x_1, \dots, x_{n-1}, \bar u $. One can prove \eqref{new181} for example by,
(a) approximating $ f ( \nabla \bar u ) $ in the $ W^{1,2} ( B ( x' , \rho ) ) $ norm by a sequence, $ ( q_j ) $ of $ C^\infty ( \rn{n} ) $ functions,  (b) using the chain rule and change of variables theorem to show that $ H_j  =  q_j  \circ \Ph \in W^{1,2} ( B ( \Psi (x'), \rho' ) ) $ with  norms bounded by a constant independent of $j, $ (c) showing that $ H_j \rar H $ in the norm of $ W^{1,2} ( B ( \Psi (x'), \rho' ) ). $ 

From \eqref{new181} and well known properties of Sobolev functions it follows that $ H $ is ``absolutely continuous on most lines''. Therefore, in our situation, if $\hat z = ( \Psi_1 (x'), \dots, \Psi_{n-1} (x') ), $ then for almost every $ t $ with $ | t - \Psi_n ( x' ) | < \rho'/2$ it is true that in a neighborhood of $ \hat z, $ we have $H(\cdot, t )\in W^{1,2} $ as a function of $ x_1, \dots, x_{n-1}. $ Let 
\[
\hat e_i =(0, \ldots, 0, 1, 0, \ldots, 0,  - \bar u_{x_i}/ \bar u_{x_n} ( x_1, \dots, x_{n-1}, t ))
\]
denote the vector with  1 in the $i$ th position and $  - \bar u_{x_i}/ \bar u_{x_n} ( x_1, \dots, x_{n-1}, t ) , $ in the $n$ th position, for $ 1 \leq i \leq n - 1, $. Then from either \eqref{new16} or \eqref{new17} we have for $\rh^{n-1}$ almost every  $(x_1, \dots, x_{n-1})$ in a neighborhood of $\hat z$ that 
\begin{align} 
\label{new19} 
\frac{ \ar H}{ \ar x_i } ( \cdot, t )  =  \nabla f ( \nabla \bar u ) \cdot  \hat e_i = 0. 
\end{align}   Transferring this information to $ f ( \nabla \bar u ) $ we see first for almost every $t$  that $ f ( \nabla \bar u ) $ is constant on  $ \{ x : \bar u ( x ) = t \} \cap B ( x', \rho/2 ). $ Second  from continuity of $ f ( \nabla \bar u )$ and $\bar u, $ we then conclude this statement for every $ t. $ Finally, the definition of a component and continuity of $ f ( \nabla \bar u )$, $\bar u$ imply assertion \eqref{new18}.

Armed with \eqref{new18} we can show for $ G' , t_0, $ as in \eqref{new6} and $ \xi_1 $ as in \eqref{new5}  that if $ O'  $ is an open set in $ \Om $ containing $ \hat K =  \{ x \in G'  : \xi_1/2 \leq  \bar u (x)  \leq t_0 \}$ then
\begin{align}  
\label{new20}  
\nu ( O' )   > 0 . 
\end{align}  
Indeed otherwise, by our construction, $ n $ homogeneity of $ f$, \eqref{new6}, and \eqref{new18} we have $ \nabla \bar u = 0 $ on  $ \ar G' \cap \{ x: \bar u ( x ) =  t_0\} $  which easily leads to a contradiction by a barrier argument.
 In fact, if 
 \[
 B ( y, \hat r ) \subset \hat K \cap \{ u < t_0 \}\; \mbox{with}\; \hat y \in \ar B ( y, \hat r ) \cap \{ x:\; \bar u ( x ) = t_0 \}
 \]
 then from the Hopf maximum principle $ | \nabla \bar u ( \hat y ) | >  0. $  From this contradiction we conclude that
\eqref{new20} is valid when $ p  = n  \geq 3. $

\subsection{The case $n=p=2$} 
In this case we note from \eqref{prooff} and the computation in Lemma \ref{v'issolnlemma} that
 $ \ti L \bar v =   0  $ weakly on $ \{ x: \nabla \bar u ( x ) \not = 0\} $ and $ \ti L $ is uniformly elliptic where
\[
 \ti L \bar v =  \sum_{k, j= 1}^2 \frac{\ar}{\ar x_k} ( f_{\eta_k \eta_j}  \frac{ \ar \bar v }{ \ar x_j}  ) 
\] 
 as in section \ref{intro}. To analyze this  case let $\hat x\in\Omega$ be any point with $\nabla \bar u ( \hat x ) = 0. $  We temporarily  use complex notation and write  $ \bar u_z = (1/2)  (\bar u_{x_1}  - i \bar u_{x_2} ) $ where $ i =
  \sqrt{ - 1}. $ We note that $ \bar u_z $ is a $ k-$quasiregular mapping of $ \Om $, where $ k = k ( p, n, c_{*}) $ (see \cite[16.4.3]{AIM09} for this fact and more on quasiregular mappings in the plane). From properties of quasiregular mappings we see that the zeros of $ u_z $ in $ \Om $  are isolated.  Next we note from the factorization theorem for quasiregular mappings (see \cite[Corollary 5.5.4]{AIM09}) that  $  \bar u_z =  \mathfrak{t} \circ \mathfrak{s} $ where $ \mathfrak{t}$ is analytic in $  \mathfrak{s}( \Om ), $  $ \mathfrak{s}$ is  a  quasiconformal mapping of  $ \rn{2}, $ and $ \mathfrak{s} ( \hat x ) = 0 . $  From  local properties of analytic functions, and $ \rn{2} $ quasiconformal mappings, as well as \eqref{prooff},  it follows  that
  there exists $ \ti r > 0 $ such that $ B ( \hat x, 8 \ti r ) \subset \Om  $  and if $ 0 < \rho \leq 2 \ti r, $ then
  \begin{align}  
  \label{new21}  
  0 < c_-^{-1}    f ( \nabla \bar u ( x ))  \leq  f ( \nabla \bar u (y) )   \leq c_-
 f ( \nabla \bar u (x) )
 \end{align} 
whenever $x, y \in B ( \hat x,  2 \rho ) \sem \bar B( \hat x, \rho/4)$. Here $ c_- \geq 1 $  may  depend on various quantities but is independent of $ \rho $ .  Using  \eqref{new21}, standard Caccioppoli type estimates for linear divergence form
PDE, and H\"{o}lder's inequality  we find that
\begin{align} 
\label{new22} 
\begin{split}
  c^{-1}  \left( \int_{ B ( \hat x, \rho ) \sem \bar B ( \hat x, \rho /2 )} | \nabla \bar v | \rd x \right)^2 & \leq  \rho^2
 \int_{ B ( \hat x, \rho ) \sem \bar B ( \hat x, \rho /2 )} | \nabla \bar v |^2 \rd x  \\
 & \leq  c  \int_{ B ( \hat x, 2 \rho ) \sem \bar B ( \hat x, \rho/4 )} | \bar v - \bar v ( \hat x  + ( \rho, 0) )  |^2 \rd x \\
 & \leq c^{2 } \rho^2
 \end{split}
 \end{align} 
 where again $ c \geq 1 $ is a positive constant independent of $ \rho. $  Putting $ \rho = 2^{-l} \ti r $ in \eqref{new22} and summing over $ l = - 1, 0, \dots, $ we find that
\begin{align} 
\label{new23}  
\int_{ B ( \hat x, 2 \ti r ) } | \nabla \bar v  | \rd x  < c^* \ti r <  \infty.   
\end{align} 

In view of \eqref{new21}-\eqref{new23} we can now use a more or less standard argument to show that if $ 0 \leq \chi \in C_0^\infty ( B ( \hat x, 2 \ti r ))$ then
\begin{align} 
\label{new24}  
\int_{ B ( \hat x, 2 \ti r ) } \sum_{k, j = 1}^2 \, f_{\eta_k \eta_j} (\nabla \bar u ) \, \frac{ \ar \bar v }{ \ar x_j} \, \frac{\ar \chi}{\ar x_k} \rd x = -  \hat a \,  \chi ( \hat x )  
\end{align}  
for some $ \hat a > 0. $
For completeness we give the proof of \eqref{new24} here. To do this let $ \si \in  C_0^\infty ( B ( \hat x, 2 \ti r ))$ with $ \si = 1 $ on $ \bar B ( \hat x, \ti r ).$ If $  \ph \in C^\infty ( B ( \hat x, 2 \ti r ) \sem \bar B ( \hat x, \rho/2) )$ we first show that for  $\rh^1$ almost every $ \rho$ with $0 < \rho < \ti r,  $
\begin{align} 
\label{new25} 
 \int\limits_{ B ( \hat x, 2 \ti r ) \sem \bar B ( \hat x, \rho) } \sum_{k, j = 1}^{2} f_{\eta_k \eta_j} (\nabla \bar u ) \frac{ \ar \bar v }{ \ar x_j} \frac{\ar ( \ph \si ) }{\ar x_k} \rd x
= \int\limits_{ \ar B ( \hat x, \rho) } \sum_{k, j = 1}^2 f_{\eta_k \eta_j} (\nabla \bar u ) \frac{ \ar \bar v }{\ar x_j} \xi_k\, \ph \rd\rh^1 
\end{align}   
where $ \xi = ( \xi_1, \xi_2 ) $ denotes the inward unit normal to $ \ar B ( \hat x, \rho). $  To verify \eqref{new25} for small $ \hat \de > 0, $  let  $ \psi \in C_0^\infty ( [\rho - \hat \de, \infty ) ) $ with $ \psi \equiv 1 $ on  $ [ \rho,  \infty ). $ Put $ \hat \psi ( x ) = \psi ( | x - \hat x | ), x \in \rn{2}, $ and  replace $ \ph\si $ by $ \ph \hat  \psi \si  $
on the left hand side of  \eqref{new25}.  Then the resulting integral is now zero since $ \ti L \bar v = 0 $ weakly in $ B ( \hat x, 2 \ti r ) \sem \{\hat x\}. $
Using this fact, the Lebesgue differentation theorem, letting $ \hat \de \rar 0, $  and  doing some arithmetic, we eventually obtain \eqref{new25}.
 Next from \eqref{new22} and a weak type estimate we see there exists $ \rho' $ with $\rho/2 \leq \rho' \leq \rho $ such that
\begin{align} 
\label{new26}  
\int_{\ar B ( \hat x, \rho' ) } | \nabla \bar v | \, \rd\rh^{1} \leq c'
\end{align} 
where $ c'$ is independent of $ \rho. $ Using \eqref{new23}, \eqref{new25}, and \eqref{new26}, we find for  a sequence $ (\rho_l ) $ with $ \lim_{l \rar \infty} \rho_l = 0 $ and  $ 0 \leq \chi \in C_0^\infty ( B ( \hat x, 2 \ti r ))$ that
  \begin{align} 
\begin{split}
  \label{new27} 
\int\limits_{ B ( \hat x, 2 \ti r )} \sum_{k, j = 1}^{2} f_{\eta_k \eta_j} (\nabla \bar u ) \frac{ \ar \bar v }{ \ar x_j} \frac{\ar ( \chi \si ) }{\ar x_k} \rd x &= 
  \lim\limits_{l \rar \infty}  \int\limits_{ B ( \hat x, 2 \ti r ) \sem \bar B ( \hat x, \rho_l) } \sum_{k, j = 1}^2  f_{\eta_k \eta_j} (\nabla \bar u ) \frac{ \ar \bar v }{ \ar x_j}
 \frac{\ar ( \chi \si ) }{\ar x_k} \rd x
     \\ 
     &   =\lim\limits_{ l  \to \infty }  \int\limits_{ \ar B ( \hat x, \rho_l ) } \sum_{k, j = 1}^2  f_{\eta_k \eta_j} (\nabla \bar u ) \frac{ \ar \bar v }{\ar x_j} \xi_k\,  \chi \rd\rh^1
\\ 
     &= \chi ( \hat x )  \lim_{ l  \to \infty }  \int_{ \ar B ( \hat x, \rho_l ) } \sum_{k, j = 1}^2 \, f_{\eta_k \eta_j} (\nabla \bar u ) \frac{ \ar \bar v }{\ar x_j} \xi_k \, \rd\rh^{1}  
\\     
    & 
    =      -  \chi (\hat x )  \hat a  
       \end{split}
      \end{align} 
for some real $\hat a. $  Now \eqref{new24} follows from \eqref{new27} and the observation that $ \chi \si $ can be replaced in \eqref{new27} by $ \chi $ since $ \chi ( 1 - \si ) $ has
compact support in $ B (  \hat x, 2 \ti r ) \sem \{ \hat x \} $ and $ \ti L \bar v = 0 $ weakly in $ B ( \hat x, 2 \ti r ) \sem \{ \hat x \}. $  Finally to show $ \hat a > 0 $ we note that \eqref{new25} remains true if $ \ph $ is replaced by $ \bar v, $ as follows from approximating $ \bar v $ in the $ W^{1,2}  ( B ( \hat x, 2 \ti r ) ) \sem \bar B ( \hat x, \rho/2) )$  norm
by smooth functions and taking limits using Lemma \ref{localholderfornablau}. Doing this we deduce from the left hand integral in \eqref{new25} that
\begin{align*} 
\int\limits_{ B ( \hat x, 2 \ti r ) \sem \bar B ( \hat x, \rho_l ) } \sum_{k, j = 1}^2  f_{\eta_k \eta_j} (\nabla \bar u ) \frac{ \ar \bar v }{ \ar x_j} \, \frac{\ar ( \bar v \si ) }{\ar x_k} \rd x
& \geq   c^{-1}  \int\limits_{ B ( \hat x,  \ti r ) \sem \bar B ( \hat x, \rho_l) } |\nabla \bar v |^2  \rd x \\
 &\hspace{1cm} -  \, c \int\limits_{ B( \hat x, 2 \ti r )\sem B( \hat x,  \ti r )  }  | \nabla \bar v | | \bar v | | \nabla \si | \rd x 
 \end{align*}
where $ c $ depends only on $ p, n, c_{*}, \al, \be. $ Moreover from the right hand integral in this inequality and \eqref{new21} we find that  
\begin{align*}
\int\limits_{ \ar B ( \hat x, \rho_l ) } \sum_{k, j = 1}^2 f_{\eta_k \eta_j} (\nabla \bar u ) \frac{ \ar \bar v }{\ar x_j} \xi_k  \, \bar v  \rd\rh^{1} = \bar v ( \hat x + ( \rho_l, 0 ) )  \int\limits_{ \ar B ( \hat x, \rho_l) } \sum_{k, j = 1}^2 f_{\eta_k \eta_j} (\nabla \bar u ) \frac{ \ar \bar v }{\ar x_j} \xi_k \,  \rd\rh^{1}+T_l  
\end{align*}
where $ | T_l  |  \leq  \bar c $ and $ \bar c $ is independent of $l.$  If $ \hat a = 0$ in \eqref{new27}, then from the above estimates it follows easily that
$ \bar v \in  W^{1,2} ( B ( \hat x, \ti r ) ). $   However  then linear elliptic PDE theory yields that $ \bar v $ is bounded in $ B ( \hat x, \ti r/2 ), $ which is a contradiction. Thus $ \hat a \neq 0. $  Using this fact and comparing the above inequalities we see that
\[ 
+ \infty  =  \lim\limits_{l \to \infty}  \bar v ( \hat x + ( \rho_l, 0 ) )  \int\limits_{ \ar B ( \hat x, \rho_l) } \sum_{k, j = 1}^2 f_{\eta_k \eta_j} (\nabla \bar u ) \frac{ \ar \bar v }{\ar x_j} \xi_k \,  \rd\rh^{1}.
\] 
Since 
\[
\bar v ( \hat x + ( \rho_l, 0 )) \to -  \infty\; \mbox{as}\; l \to \infty
\]
it follows that necessarily $ \hat a > 0. $ From \eqref{new24} we see that $ \ti L \bar v$ may be regarded weakly as an atomic measure on $ \Om $ when $ p = 2, n = 2 $ and hence
\eqref{new20} is also valid when $ n = 2, p = 2. $

We now are in a position to finish the proof of  Lemma \ref{finess} when $ p = n $ and for a general $ f, $ as well as when $ f = g^p, p \in [n, n+1], $ and $ g $ is as in Theorem \ref{alv13.2}. We consider first the case when $ f = g^p, $ as the compactness argument in either case is essentially
  the same.

 We shall need some more notation. For fixed $ \al, \be, $ let 
 \[
 \ti \Ga_m = \{\ti Q^{(m)} \},\; m = 1, 2, \dots,  
  \]
  be collections of cubes with side lengths defined as in section \ref{intro} with $ a_1, a_2, ..., $ replaced by  $a_1^{(m)}, a_2^{(m)}, \ldots,$ where $ 0 < \al \leq a_{k}^{(m)} < \be, $ for $ k, m = 1, \dots . $ Let $ \m{C}_m $ denote the corresponding Cantor set and suppose $ \ti \La_m $ is a finite covering of $ \m{C}_m $  by disjoint cubes in $ \ti \Ga_m .$ Define $
  Q^{(m)} $ relative to $ \ti Q^{(m)} $ in the same way that $ Q_0 $ below \eqref{new3} is defined relative to $ \ti Q_0 $ and set  
  \[
   \La_m := \{ Q^{(m)} :\; \ti Q^{(m)} \in \ti \La_m \}\; \mbox{and}\; \Om_m := B (0, n ) \sem \bigcup_{ Q^{m} \in \La_m} \bar Q^{(m)}.
    \]
Suppose $ (p_m) $ is a sequence of points in $ [n, n + 1] $ with $ \lim_{m \to \infty} p_m = \hat p. $  Let $ f_m = g^{p_{m}} $ and  let $ \bar u_m $ be the weak solution to \eqref{flaplace} relative to $f_m $ on $ \Om_m $ with continuous boundary values, 1 on $ \ar B (0, n ) $ and 0 on $ \ar Q^{(m)}$ for every $Q^{(m)} \in \La_m$. Extend $  \bar u_m $ to $ B (0, n ) $ by putting $ \bar u_m = 0 $ on $Q^{(m)}$ for every $Q^{(m)}\in \La_m$. Let $\bar \mu_m$ be the measure associated with $ \bar u_m $ as in \eqref{ast} and let $ \bar v_m = \log
 f _m ( \nabla \bar u_m ) . $
 Finally define  $ \ti L_m \bar v_m = \nu_m, $ weakly as in
Lemma \ref{v'issolnlemma} relative to $ \bar u_m, f_m, $ on $ \{ x : \nabla u_m (x)  \not = 0 \} $ when $ n \geq 3 $ and on $ \Om_m$ when $ n
= p = 2 $ (see \eqref{new24}).

From \eqref{new20} and \eqref{new14}  we see that if Lemma \ref{finess} is false for $c_4$ sufficiently large and $ n \geq 3$, then there exists $ \ti Q^{(m)} \in \ti \Ga_m $ with
\begin{align} 
\label{new28}  
0 <  \int_{O_m \cap \left\{ | \nabla \bar u_m | > 0 \right\} } \bar u_m \rd \nu_m  =  b_m  \mu_m (\bar Q^{(m)} )
\end{align} 
  where 
  \[
  O_m = \left\{x \in  (1 +\he) \ti Q^{(m)}:\;  d ( x, \ar \Om_m ) \geq
 \frac{ s(\ti Q^{(m)} )}{c_4} \right\}\;  \mbox{and}\; 0 < b_m \to 0\; \mbox{as}\; m \to \infty. 
 \]
Let  $ z_m $ denote the center of $ \ti Q^{(m)} $ and let $ \hat \Om_m =  \{ y:  z_m  + s(\ti Q^{m}) y  \in \Om_m \} $. Put
\begin{align*} 
\hat u_m ( y ) = \frac{\bar u_m ( z_m + s(\ti Q^{m}) y )}{ \max\limits_{ 2 \ti Q^{(m)} }  \bar u_m }\; \mbox{whenever} \; y\in \hat \Om_m . 
\end{align*} 
Using translation and dilation invariance of \eqref{prooff} we see that $ \hat u_m  $ is a weak solution to \eqref{flaplace} in $\hat \Om_m. $ Let $ \hat \mu_m $  denote the measure corresponding to $ \hat u_m$ with $ f$ and $u$ replaced by $f_m$ and $\hat u_m$. Then from Lemma \ref{uismu} and Harnack's inequality we find from estimates similar to those in \eqref{new1} that
\begin{align} 
\label{new30} 
c^{-1} \leq \hat \mu_m ( \m{S} ) \leq  \max_{1000 \m{S}} \hat u_m  \leq c \hat \mu_m  ( 2000 \m{S} ) \leq c^2 
\end{align} 
where $ c \geq 1 $ is independent of $ f_m, p_m \in [n, n+1] $ for fixed $ c_*$ in \eqref{prooff}. Once again we emphasize that this independence follows from the fact that the constants in Lemmas \ref{fislegthanu}-\ref{localholderfornablau} can be chosen independent of these quantities.  Let $ \hat v_m = \log f_m ( \nabla \hat u_m ). $  Then  $ \hat v_m  $ is a weak sub solution to $\hat L_m $ in the interior of  $ 1000 \m{S} \cap \{ x : \nabla \hat u_m \not = 0 \} $  where $ \hat L_m $ is defined as in Lemma \ref{v'issolnlemma} relative to $ \hat u_m, f_m. $ Let $ \hat \nu_m $ be the corresponding measure. From \eqref{new28}-\eqref{new30} we deduce that if $ \hat O_m = \{ y : z_m + s ( \ti Q_m ) y \in O_m \}, $ then
\begin{align} 
\label{new31} 
\int_{\hat O_m } \hat u_m  \rd \hat \nu_m \to 0\; \mbox{as}\; m \to \infty. 
\end{align} 
Using \eqref{new30},  Lemmas \ref{fislegthanu}-\ref{localholderfornablau}, the fact that  $ d ( \cdot, \ar \hat \Om'_m ) $ is Lipschitz,
 and  Ascoli's theorem we see there exists  sub sequences, $ ( \hat \Om_m' ) $ of $ ( \hat \Om_m ) $ such that  $ \hat \Om'_m \cap B ( 0, R ) $ converges to $ \hat \Om \cap B ( 0, R ) $ for each $ R > 0 $ in the Hausdorff distance metric and  $ ( \hat u_m' ) $ of $ (\hat u_m ) $ with $ ( \hat u_m' ) $  converging uniformly to $ \hat u $ in  the interior of $ 1000 \m{S}. $  We also can choose the sub sequence so that  $ \nabla \hat u'_m  $ converges uniformly to $ \nabla \hat u$ on compact subsets of $\hat \Om.$ Using these facts it is easily seen that $ \hat u $ is a weak solution to \eqref{flaplace} with $ \hat f = g^{\hat p} $ in the interior of $ 1000 \m{S} \cap \hat \Om $  and  $ \hat u $ is continuous in the interior of $ 1000 \m{S} $ with $ \hat u = 0 $ on $ \ar \hat \Om \cap 1000 \m{S}. $ Let $ \hat \mu $ be the measure corresponding to $ \hat u $ and let $ \hat \nu $ be the measure corresponding to $ \hat v. $ Then for $ n \geq 3 $ we may also assume that  $  \hat L'_m \hat v
 '_m = \hat \nu_m' $  converges weakly to $ \hat L \hat v = \nu $ as measures on compact subsets in the interior of $ \m{S} \cap \{ x : \nabla \hat u ( x ) \neq 0\}. $ Indeed from the definition of $ f_m $ and uniform convergence of $ (\nabla \hat u'_m ) $  we see that $ ( f_{m} )_{\eta_k \eta_j} ( \nabla \hat u'_m ), 1 \leq k, j \leq n, $ converges uniformly on compact subsets in the interior of
$ \m{S} \cap \{ x : \nabla \hat u ( x ) \neq 0\}. $ Also from Lemma \ref{localholderfornablau} we deduce that for large $ m, $ \, $ \hat v'_m $ is uniformly bounded in $ W^{1,2} $ on an open set with compact closure in
 $ \m{S} \cap \{ x : \nabla \hat u ( x ) \neq 0\}. $ Using these facts and well known theorems on weak convergence in $ W^{1,2} $ we see that if $ n \geq 3, $ then a sub sequence of $ ( \hat v'_m ) $ (also denoted $(\hat v'_m ))$ yields,
\begin{align} 
\label{new32} 
\begin{split}
-  \lim_{m \to \infty} \int \ph \rd \hat \nu'_m &= \lim_{m\to \infty} \int \sum_{j,k=1}^{n} (f_m)_{\eta_k \eta_j} ( \nabla \hat u'_m ) (\hat v'_m)_{x_j} \ph_{x_k} \rd x  \\
&= \int \sum_{j,k=1}^{n} \hat f_{\eta_k \eta_j} ( \nabla \hat u ) \hat v_{x_j} \ph_{x_k} \rd x \\
&= -\int \ph d \hat \nu 
\end{split}
\end{align} 
whenever $ \ph $ is infinitely differentiable with compact support in  $ 1000\m{S} \cap \hat \Om \cap \{ x : \nabla \hat u ( x ) \neq 0 \} $.
If  $ n = 2 $ we claim that $ \hat \nu'_m $ converges weakly to $ \hat \nu $ on compact subsets in the interior of $1000 \m{S} \cap \hat \Om . $ To see this we note from
 the discussion preceding \eqref{new21} that there exists $ t_m $ analytic in $ s_m ( \hat \Om'_m ) $  and $ s_m $ quasiconformal in $ \rn{2} $ with  $ ( u_m')_z = t_m \circ s_m $ in $  \hat \Om_m' $.  From normal family type arguments for $ \rn{2} $ quasiconformal mappings and analytic functions we see that there exist subsequences of $ (t_m), (s_m) $ (also denoted $ (t_m), (s_m) $ ) with $ (s_m) $ converging to $s$ a quasiconformal mapping of $ \rn{2}, $ uniformly on compact subsets of $ \rn{2}, $ and $ t_m $ converging uniformly to $ t $ analytic, uniformly on compact subsets in the interior of  $ s ( 1000 \m{S} \cap \hat \Om).$ Using these facts and the argument principle for analytic functions we conclude that the constants in \eqref{new21} - \eqref{new23} can be chosen independent of $m. $ From this conclusion, uniform convergence of $ ( \nabla \hat u'_m)$ and simple estimates in \eqref{new24} we obtain \eqref{new32} for $ \ph $ infinitely differentiable with compact support in
  $1000\m{S} \cap \hat \Om$. Let 
\[
 \hat O = \left\{ x \in (1+\he ) \m{S}:\; d ( x, \ar \hat \Om ) > c_4^{-1} \right\}. 
 \] 
 Then from \eqref{new32} and \eqref{new31} we have
\begin{align} 
\label{new33} 
\hat \nu (\hat O) = 0. 
\end{align} 
On the other hand,  we can essentially repeat the argument from \eqref{new6}-\eqref{new27} since the same constants in Lemmas \ref{fislegthanu}-\ref{localholderfornablau} as earlier can also be used for $ \hat u. $  Moreover,  since $ 1000 \m{S} \cap \hat \Om_m $ converges in the Hausdorff distance sense to $ 1000 \m{S} \cap \hat \Om $ the Harnack chains used to obtain
the analogue of \eqref{new20} can all be chosen in $ \hat \Om_m $ for $ m $ large enough. A more cut to the chase type argument is to observe that if $ \hat x_m, t_0^{(m)}, G_m' $ denote the sets in \eqref{new6}, and $ \xi_1^{m} $ is as in \eqref{new5} relative to
$ \hat u_m' $ in $ (1+\he) \m{S} \cap \hat \Om'_m,$  then these sequences converge pointwise and in the Hausdorff distance sense to $ \hat x', \hat t_0, \hat \xi_1 , \hat G'  \subset (1+\he) \m{S} .$ Moreover \eqref{new5},\eqref{new6} are now valid for $ \hat u $ in this symbology. Repeating the argument leading to \eqref{new14} we see that in order to avoid a contradiction to \eqref{new33} we must have $ \hat p = n. $  Now repeating the argument from \eqref{new13} to \eqref{new27} we also rule out the case $ \hat p = n $ and so for
$ c_2, c_4 $  large enough, obtain  $ \hat \nu ( \hat O ) > 0, $ a contradiction to  \eqref{new33}.  The proof of Lemma \ref{finess} is now complete when $ f = g^p$. For a general $ f $ it follows from \eqref{new14} that we need only consider the case $ p = n. $ If $ p = n, $ we again argue by contradiction and use a compactness argument similar to the above to get a contradiction. We omit the details.  
\end{proof}

Following \cite[Chapter IX, Theorem 2.1]{GM05},  we continue the proof of Proposition \ref{prop2} by repeating the stopping time argument in Theorem \ref{alv13.1} only with cubes in $ \ti \Ga $ rather than balls.  First let $ M > > 1 $ be so large that if $ \ti Q \in \ti \Ga $ and $ \mu (\ti Q ) \geq M s ( \ti Q )^{n-1}, $ then
 $ s ( \ti Q ) \leq \min ( \tau, 10^{-5} ). $ This choice is possible as we see from \eqref{new2}. Let $ s << \tau $ and choose a covering $ \ti \La_{M} = \m{B}_M \cup \m{G}_M $ of  $ \m{C} $ by cubes in $ \ti \Ga , $ according to the following recipe. Either $ x \in \m{C} $ lies in a cube in
\[ 
\m{G}_{M}:=\left\{\ti Q \in \ti \Ga:\; s ( \ti Q ) > s, \; \mu(\ti Q) \geq M s(\ti Q)^{n-1}, \; \mbox{and}\; \ti Q \; \mbox{is maximal}\right\} 
\]
or  no such cube exists and  $x$ lies in a cube in
\[ 
\m{B}_{M}:=\{\ti Q \in  \ti \Ga:\; \, s ( \ti Q ) \leq s\;  \mbox{and}\; \ti Q\; \mbox{is maximal}\}.
\] 
Note that $\m{B}_{M}\cup\m{G}_{M}$ is a disjoint covering of $\m{C}$. As earlier let $ \La_{M} = \{ Q : \ti Q \in \ti \La_{M} \} $ and define
$ \bar u $ as below \eqref{cubesaredisjoint} relative to $ \La_{M}. $ Then $ \bar u $ is a solution to \eqref{flaplace} in $ \Om = B (0, n ) \sem \cup_{Q \in \La_{M}} \bar Q $ and continuous in $ B (0, n ) $ with $ \bar u = 0 $ on $ \cup_{Q \in \La_{M}} \bar Q $ while $ \bar u = 1 $ on $ \ar B (0, n ) . $ Let $\bar \mu$ be the measure associated with $\bar u$ as in \eqref{ast}. From the maximum principle for solutions to \eqref{flaplace} we see that $ \bar u \leq u $ in $ \m{S} $ and as in \eqref{new1} and \eqref{new2} that for $ \ti Q \in \ti \La_{M}, $
\begin{align} 
\label{new34} 
\begin{split}
&s (\ti Q )^{1-n} \bar \mu ( \bar Q ) \leq c \, s ( \ti Q )^{1-p} \max\limits_{(1+\he) \ti Q} u^{p-1} \leq c^2\,  s( \ti Q )^{1-n} \, \mu ( 2 \ti Q )\\
&\bar \mu ( B (0, n )) \approx 1.
\end{split} 
\end{align} 
where $Q\in\La_{M}$  corresponds  to $\ti Q \in  \ti \La_M$. Let 
\[ 
\m{E}:= \left\{ \ti Q \in \ti \Ga \sem \ti \La_{M}\; \mbox{for which there exists}\; \ti Q' \in \ti \La_{M}\;  \mbox{with}\; \ti Q' \subset \ti Q\; \mbox{and}\; c_2 \, s ( \ti Q' ) \leq  s ( \ti Q )\right\}. 
\] 
For $ c_2, c_4 $ as in Lemma \ref{finess} and $ \ti Q \in \m{E} $ we also define
\[ 
O := O ( \ti Q ) = \left\{x \in (1 + \he ) \ti Q\; \mbox{with}\; d(x, \ar \Om ) \geq \frac{ s(\ti Q)}{c_4} \right\}. 
\] 
We note that each point in
\[ 
\bigcup_{ \ti Q \in \m{E} } O ( \ti Q )\; \mbox{lies in at most}\; \hat N\; \mbox{of the}\; \ti Q \in \m{E}
\]
where $ \hat N $ has the same dependence as $ c_4. $ Using this observation and Lemma \ref{new2} it follows for $ n \geq 3 $ that
\begin{align} 
\label{new35}  
\sum_{\ti Q \in \m{E} } \bar \mu ( \bar Q ) \leq \breve{c} \int\limits_{ \Om \cap \left\{x : \nabla \bar u \neq 0 \right\}}  \bar u \rd \nu \leq
2 \breve{c} \int\limits_{ \Om \cap \left\{ x : |\nabla \bar u |  >  \de''  \right\}}  \bar u \rd \nu. 
\end{align}  
provided
$ \de'' > 0 $ is small enough. 

If $ p = 2 = n $ the integral on the right hand side of \eqref{new35} is taken over $ \Om . $ In general $ \breve c $ depends on $ p, n, \al, \be, c_{*}$ but in view of Lemma \ref{new2} we have $ 1 \leq \breve c \leq c ( p - n )^{-1}, $ where $ c $ can be chosen to depend only on $ n, \al, \be, c_* $ when $
p \in [n, n + 1] $ while if $ f = g^p $ then $ \breve c $ can be chosen to depend only on $ n, \al, \be, c_* $ when $ p \in [n, n+1]. $

We now essentially repeat the argument leading to Lemma \ref{logfgradu}. Choose  $ \eta \in (-\infty, \infty ) $ so small that if $ | \xi |  \leq \de''$ then
 $ \log f ( \xi ) \leq \eta. $ Using \eqref{new35} and arguing as \eqref{3.19}-\eqref{3.30}  we obtain for $ n \geq 3 $ and
 $ v'  = \max (\bar v, \eta ) $ that
\begin{align} 
\label{new36} 
\sum_{\ti Q \in \m{E} } \bar \mu ( \bar Q ) \leq 2 \breve c \int\limits_{\Om \cap \left\{ x : | \nabla \bar u | > \de'' \right\} }  \bar u \rd \bar \nu \leq - \int_{\Om}  \sum_{k, j = 1}^{n} f_{\eta_k \eta_j } ( \nabla \bar u ) v'_{x_j} \bar u_{x_k} \rd x \leq  c \breve c \log M.
  \end{align} 
To estimate the left hand side of \eqref{new36}, given $ \ti Q' \in \ti \La_{M}, $  we let
$ \si ( \ti Q' ) $ be the number of cubes $ \ti Q  \in \m{E} $ with
$ \ti Q' \subset \ti Q $ and $ c_2 s ( \ti Q' ) \leq s (\ti Q ). $  From our construction we see for $ \tau $ small enough that
\begin{align} 
\label{new37}  
\si ( \ti Q' )  \geq - c^{-1}  \log ( s ( \ti Q' ) )  
\end{align}  
From \eqref{new36} and \eqref{new37} we get
\begin{align} 
\label{new38} 
- \, \sum_{ \ti Q' \in \ti \La_{M}}  \log (s ( \ti Q'))  \bar \mu ( \bar Q' )  \, \leq \,
 c  \sum_{\ti Q \in \m{E} }  \bar \mu ( \bar Q ) \leq c^2  \breve c   \log M \, 
 \end{align} 
 where $ c \geq 1 $ in \eqref{new36}, \eqref{new37}, and \eqref{new38} has the same dependence as $ c_2 $ in Lemma \ref{finess}.
 From \eqref{new34} and \eqref{new38}  we see that if $ c $ is large enough with  
 \[
 \ti \La_1 := \{ \ti Q \in \ti \La_{M} : s ( \ti Q ) \leq  M^{ - c^3 \breve c} \}\; \mbox{and}\; \La_1 := \{ Q : \ti Q \in \ti \La_1 \}
 \]
 then
\begin{align} \label{new39} 
 \sum_{ Q \in \La_1 }  \bar \mu ( \bar Q ) \leq  (1/2) \ti \mu ( B (0, n ) ). 
\end{align} 
Finally  choosing   $s << \min (M^{ - c^3 \breve c}, \tau ), $  we see that $ \m{B}_{M} \subset  \ti \La_1. $
Let
\[  
F :=  \m{C} \cap \left( \bigcup_{  \ti Q \in \ti \La_{M} \sem \ti \La_1} \ti Q \right) . 
\] 
Then from \eqref{new1}, \eqref{new34}, and \eqref{new39} we deduce for $ c$ having the same dependence as in  \eqref{new36}-\eqref{new39} that
\begin{align} 
\label{new40} 
c^{-1} \leq \bar \mu \left( \bigcup_{Q \in \La_{M} \sem \La_1} \bar Q \right) \leq c \mu ( F) . 
\end{align} 
Moreover, if $ \de'  =  \frac{1}{2 \breve c c^3}, $ where $ c $ is as in the definition of $  \ti \La_1,$ then since $ \ti \La_{M} \sem \ti \La_1 \subset \m{G}_M, $ we have
\begin{align} 
\label{new41}    
\sum_{\ti Q \in \ti \La_{M} \sem \La_1 }  s ( \ti Q )^{n - 1 - \de' }  \leq c M^{- 1/2}  \sum_{ \ti Q \in \ti \La_{M} \sem \La_1 } \mu ( \ti Q ) \leq M^{ - 1/4} \leq \ep 
\end{align} 
provided $ M \geq M_0 $ is large enough. In view of our earlier calculations we conclude  that $ \de' $ has the same dependence as in \eqref{new3}.
 Moreover if $ f = g^p, g $ as in Theorem \ref{alv13.2}, then $ M_0 $ can be chosen independent of $ p $ in $ [n, n + 1].
 $ It follows from \eqref{new40} and \eqref{new41} that \eqref{new3} is true. From our earlier remarks we conclude that Proposition \ref{prop2} holds which finishes proof of Theorem \ref{alv13.2}.
\end{proof}
\section*{Acknowledgment}
The authors would like to thank Matthew Badger reading an earlier version of this manuscript and for his suggestions. The first and second authors were partially supported by NSF DMS-0900291 and by the Institut Mittag-Leffler (Djursholm, Sweden). Both authors would like to thank the staff at the institute for their gracious hospitality. The first author has also been supported in part by ICMAT Severo Ochoa project SEV- 2011-0087. He acknowledges that the research leading to these results has received funding from the European Research Council under the European Union's Seventh Framework Programme (FP7/2007-2013)/ ERC agreement no. 615112 HAPDEGMT.

\def\cprime{$'$} \def\cprime{$'$}
\providecommand{\bysame}{\leavevmode\hbox to3em{\hrulefill}\thinspace}
\providecommand{\MR}{\relax\ifhmode\unskip\space\fi MR }
\providecommand{\MRhref}[2]{%
  \href{http://www.ams.org/mathscinet-getitem?mr=#1}{#2}
}
\providecommand{\href}[2]{#2}

\end{document}